\documentclass{amsart}
\usepackage{amssymb, latexsym, amscd}
\usepackage[mathscr]{eucal}
\usepackage{enumerate}
\usepackage{fullpage}

\usepackage[pdftex]{graphicx, color}
\usepackage[roman]{sublabel}

\definecolor{link_red}{rgb}{0.7,0,0}
\definecolor{cite_blue}{rgb}{0,0,0.97}
\usepackage[colorlinks=true,citecolor=cite_blue,linkcolor=link_red]{hyperref}

\usepackage{ottmacros}

\theoremstyle{plain}
\newtheorem{theorem}{Theorem}[section]
\newtheorem{corollary}[theorem]{Corollary}
\newtheorem{lemma}[theorem]{Lemma}
\newtheorem{proposition}[theorem]{Proposition}

\newtheorem{m_theorem}{Theorem}

\theoremstyle{definition}
\newtheorem{definition}[theorem]{Definition}

\newtheorem{remark}[theorem]{Remark}

\newcommand{\beq}{\begin{equation}}
\newcommand{\eeq}{\end{equation}}
\newcommand{\beqn}{\begin{equation*}}
\newcommand{\eeqn}{\end{equation*}}

\newcounter{arealm}
\newcounter{grealm}
\newcounter{irealm}
\newcounter{mrealm}
\newcounter{sarealm}

\renewcommand{\thearealm}{A\arabic{arealm}}
\renewcommand{\thegrealm}{H\arabic{grealm}}
\renewcommand{\theirealm}{I\arabic{irealm}}
\renewcommand{\themrealm}{M\arabic{mrealm}}
\renewcommand{\thesarealm}{SA\arabic{sarealm}}

\numberwithin{equation}{section}

\begin{document}

\title{From limit cycles to strange attractors}

\author{William Ott}
\address[William Ott]{Department of Mathematics, University of Houston, Houston, 
TX 77204-3008, USA.}
\email{ott@math.uh.edu}
\urladdr{http://www.cims.nyu.edu/$\sim$ott}

\author{Mikko Stenlund}
\address[Mikko Stenlund]{
Courant Institute of Mathematical Sciences\\
New York, NY 10012, USA; Department of Mathematics and Statistics, P.O. Box 68, Fin-00014
University of Helsinki, Finland.} 
\email{mikko@cims.nyu.edu}
\urladdr{http://www.math.helsinki.fi/mathphys/mikko.html}



\date{\today}

\begin{abstract}
  We define a quantitative notion of shear for limit cycles of flows.  We prove that
  strange attractors and SRB measures emerge when systems exhibiting limit cycles with
  sufficient shear are subjected to periodic pulsatile drives.  The strange attractors
  possess a number of precisely-defined dynamical properties that together imply chaos
  that is both sustained in time and physically observable.
\end{abstract}

\maketitle

\subsection*{Acknowledgements}
Mikko Stenlund was partially supported by the Academy of Finland. William Ott has 
been partially supported by NSF grant DMS-0603509.

\section{Introduction}\label{s:intro}

This paper is about a mechanism for producing chaos: shear.  We are guided by the idea
that in the presence of shear, a stable dynamical structure can be transformed into a
strange attractor with strong stochastic properties by forcing the structure with a
pulsatile drive.  The forcing does not overwhelm the intrinsic dynamics.  Instead, it acts
as an amplifier, amplifying the effects of the intrinsic shear.  We focus on one
particular dynamical structure of great importance: the limit cycle.  Limit cycles are
asymptotically stable periodic orbits of flows on Riemannian manifolds.

The application of a periodic pulsatile drive to a flow exhibiting a limit cycle causes
deformations to occur.  If shear is present in a neighborhood of the limit cycle, if the
limit cycle only weakly attracts nearby orbits, and if the time between pulses (the
relaxation time) is sufficiently large, then stretch-and-fold geometry emerges in a
neighborhood of the limit cycle.  Stretch-and-fold geometry suggests that chaotic behavior
that is both sustained in time and observable {\itshape may} exist.  We prove that such
chaotic behavior does exist in a certain parameter regime for {\itshape any} (generic)
forcing function if the shear is sufficiently strong.  Moreover, we define a quantity
called the shear integral that quantifies the amount of shear that is present in the
intrinsic flow in a neighborhood of the limit cycle.  We emphasize that the shear integral
depends only on the intrinsic system and not on the external forcing.  Our result is the
first of its kind for general limit cycles.  Wang and Young~\cite{WqYls2002, WqYls2003}
obtain results of a similar flavor for supercritical Hopf bifurcations and certain linear
models.

The search for and analysis of stochastic behavior in deterministic dynamical systems have
played a major role in guiding dynamical systems research.  We discuss a few relevant
developments.  The theory of uniformly hyperbolic systems is well-developed.  Let $M$ be a
compact Riemannian manifold and let $f : M \to M$ be a $C^{2}$ diffeomorphism of $M$.  An
attractor for $f$ is a compact set $\Om$ satisfying $f(\Om)=\Om$ for which there exists an
open set $U \subset M$ (the basin) such that $f(\bar{U}) \subset U$ and $\Om =
\bigcap_{i=0}^{\infty} f^{i} (\bar{U})$.  An attractor $\Om$ is said to be an Axiom A
attractor if the tangent bundle over $\Om$ splits into $2$ $Df$-invariant subbundles
$E^{s}$ and $E^{u}$ such that vectors in $E^{s}$ are contracted by $Df$ and vectors in
$E^{u}$ are expanded by $Df$ (we assume $E^{u}$ is nontrivial).  An Axiom A attractor
supports a special invariant measure known as a Sinai-Ruelle-Bowen (SRB) measure that
describes the asymptotic distribution of the orbit of almost every point in $U$ with
respect to Riemannian volume and has strong stochastic properties.  In this sense, the
chaotic behavior associated with Axiom A systems is observable.  It is also sustained in
time because of the presence of positive Lyapunov exponent(s).  One can, in principle,
detect the presence of uniform hyperbolicity in a given system by finding invariant cone
families with suitable properties.  For example, Tucker uses this approach to prove that
the Lorenz equations are chaotic for the classical parameter values studied by
Lorenz~\cite{Tw2002}.

Many systems of interest in the biological and physical sciences display some form of
hyperbolicity but are not uniformly hyperbolic.  A mature theory of nonuniform
hyperbolicity has emerged over the last $4$ decades.  However, the following problem
remains a challenge.  Given a dynamical system (or a parametrized family of dynamical
systems), how can nonuniform hyperbolicity be detected?  Numerical techniques include the
calculation of Lyapunov exponents and the $0$-$1$ test~\cite{FiGgMiWk2007, GgMi2004}.
This paper addresses the analytical component of the problem in the context of limit
cycles.  Our proofs are based on the recently-developed theory of rank one
maps~\cite{WqYls2001, WqYls2008}.  Rank one theory is based on the ideas of
Jakobson~\cite{Jm1981}, Benedicks and Carleson~\cite{BmCl1985, BmCl1991}, and
Young~\cite{Yls1998, Yls1999}.  Rank one theory provides checkable conditions that imply
the existence of SRB measures with strong stochastic properties in parametrized families
of diffeomorphisms.

We conclude the introduction with a remark that the results obtained in this paper are
in some sense dual to the phenomenon known as self-induced stochastic resonance (SISR)
(see {\itshape e.g.}~\cite{DlVEeMc2005}).  Our results demonstrate that certain intrinsic
characteristics of a deterministic system (shear) can produce stochastic-type behavior
when the system is forced in a deterministic way.  SISR demonstrates that underlying phase
space structures can produce deterministic (coherent) behavior in stochastically-forced
systems when the noise level is taken to $0$ along certain distinguished limits.

\section{Statement of results}\label{s:state_res}

We state the main results and discuss their relationship to the existing literature.  Let
$\mbi{f} : \mbb{R}^{n} \to \mbb{R}^{n}$ be a $C^{5}$ vector field and consider the
differential equation
\begin{equation}\label{e:intrinsic}
\tdf{\mbi{x}}{t} = \mbi{f} (\mbi{x}).
\end{equation}
We assume that~\eqref{e:intrinsic} admits an asymptotically stable hyperbolic periodic
solution $\bsym{\et}$ of length $L$ and period $p_{0}$.  Let $\bsym{\ga} : \mbb{R} \to
\mbb{R}^{n}$ be a function of the parameter $s$ that parametrizes $\bsym{\et}$ by length.
Define $\Ga = \{ \bsym{\ga} (s) : s \in [0, L) \}$.  Solutions to~\eqref{e:intrinsic} that
begin sufficiently close to $\Ga$ will converge to $\Ga$ at an exponential rate as $t \to
\infty$.  We are interested in the effects of adding periodic pulsatile forcing to the
vector field defining~\eqref{e:intrinsic}.  For $0 < \rh < T$, define the periodic
function $P_{\rh,T} : \mbb{R} \to \mbb{R}$ as follows.  For $0 \leqs t \leqs T$, set
\begin{equation*}
P_{\rh,T} (t) =
\begin{cases}
1, &\text{if } 0 \leqs t \leqs \rh\\
0, &\text{if } \rh < t < T
\end{cases}
\end{equation*}
and then extend periodically to all $t \in \mbb{R}$ by requiring $P_{\rh,T} (t+T) =
P_{\rh,T} (t)$.  We study the externally-forced system
\begin{equation}\label{e:forced}
\tdf{\mbi{x}}{t} = \mbi{f} (\mbi{x}) + \ve P_{\rh,T} (t) \mbi{F} (\mbi{x}) 
\end{equation}
where $\mbi{F} : \mbb{R}^{n} \to \mbb{R}^{n}$ is a $C^{4}$ vector field and the parameter $\ve > 0$
controls the amplitude of the forcing.  Notice that the right side of~\eqref{e:forced} is not
continuous.

In Section~\ref{s:singular_lim} we compute a normal form of equation~\eqref{e:forced} that
is valid in a tubular neighborhood $\tilde{M} \approx \Ga \times D$, where $D$ is a closed
disk in $\mbb{R}^{n-1}$ of sufficiently small radius.  We are interested in the dynamics
of~\eqref{e:forced} in the tubular neighborhood $M \approx \Ga \times \frac{1}{2} D$.
Since the external forcing is periodic with period $T$, it is natural to study the
time-$T$ map induced by~\eqref{e:forced}.  We write the time-$T$ map as the composition of
a {\bfseries \itshape kick map} $H_{k} : M \to \tilde{M}$ and a {\bfseries \itshape
  relaxation map} $H_{r} : \tilde{M} \to \interior (M)$.  Let $H_{k}$ be the time-$\rh$
map induced by the flow associated with~\eqref{e:forced}.  Notice that the external
forcing is active during the kick phase because $P_{\rh, T} (t) = 1$ for $0 \leqs t \leqs
\rh$.  For $\ve$ sufficiently small, $H_{k}$ maps $M$ into $\tilde{M}$ diffeomorphically.
Let $H_{r}$ be the time-$(T - \rh)$ map induced by~\eqref{e:forced} with $\ve$ set to $0$.
There exists $T_{0} = T_{0} (\ve)$ such that if $T \geqs T_{0}$, then $H_{r}$ maps
$\tilde{M}$ into $\interior (M)$.  The composition $G_{T} \defas H_{r} \circ H_{k}$ is the
time-$T$ map induced by~\eqref{e:forced}.

The dynamical properties of $G_{T} : M \to \interior (M)$ depend on a number of factors.
One feature common to every map $G_{T}$ for $T \geqs T_{0}$ is the existence of an
attractor $\Om$ defined by
\begin{equation*}
\Om = \bigcap_{i=0}^{\infty} G_{T}^{i} (M).
\end{equation*}
We call $U \defas \interior (M)$ the basin of attraction of $\Om$.  For every $\mbi{x} \in
U$, $G_{T}^{i} (\mbi{x}) \to \Om$ as $i \to \infty$.  Two characteristics of the intrinsic
system~\eqref{e:intrinsic} play a key role in determining the structure of $\Om$ and the
dynamical properties of $G_{T}$: shear and the strength of the limit cycle.  We quantify
these notions momentarily; for now, imagine that~\eqref{e:intrinsic} exhibits strong shear
in $M$ if for most points $\mbi{x} \in \Ga$, the velocity vector $\mbi{f} (\hat{\mbi{x}})$
varies substantially as $\hat{\mbi{x}}$ moves away from $\mbi{x}$ in directions orthogonal
to the limit cycle $\Ga$.  Think of the limit cycle $\Ga$ as strongly stable if solutions
to~\eqref{e:intrinsic} that begin in $M$ converge quickly to $\Ga$.  If the shear is weak
and the limit cycle is strongly stable, then the attractor $\Om$ associated with $G_{T}$
will be an invariant closed curve.  We are interested in the opposite situation.  Suppose
that the shear is strong in $M$ and the limit cycle is weakly stable.  The addition of the
periodic pulsatile external force $\ve P_{\rh, T} (t) \mbi{F} (\mbi{x})$ will amplify the
effect of the shear in the following way: disturbances that are created when $P_{\rh, T} =
1$ will be stretched during the relaxation period (when $P_{\rh, T} = 0$).  The stretching
effect increases in intensity as $T$ increases.  If $T$ is large, then folds will be
created in the phase space.  If $G_{T}$ exhibits stretch-and-fold geometry, then $G_{T}$
potentially exhibits chaotic behavior that is sustained in time and observable.

This paper aims to accomplish the following.
\begin{enumerate}[(1)]
\item
We define a computable quantity called the shear integral that quantifies the shear
associated with the intrinsic system~\eqref{e:intrinsic} near the limit cycle $\Ga$.
\item We prove that if the magnitude of the shear integral is sufficiently large and if
  the contraction near the limit cycle $\Ga$ is sufficiently weak, then the following
  holds for suitable values of $\ve$.  For a typical external vector field $\mbi{F}$,
  there exists $T_{1} > 0$ and a set $\De \subset [T_{1}, \infty)$ of positive Lebesgue
  measure such that for $T \in \De$, the time-$T$ map $G_{T}$ associated
  with~\eqref{e:forced} admits a strange attractor $\Om$ and exhibits chaos that is
  sustained in time and observable.
\end{enumerate}
The quantity $T_{1}$ satisfies $T_{1} \gg \rh$, ensuring sufficient relaxation time for
the stretch-and-fold geometry to emerge.  The term {\bfseries \itshape strange attractor}
refers to a number of precisely defined dynamical and structural properties that represent
sustained, observable chaos.  For $T \in \De$, $\Om$ supports a unique ergodic SRB measure
$\nu$.  Here the term SRB measure refers to a measure $\nu$ with a positive Lyapunov
exponent $\nu$ almost everywhere and whose conditional measures on
unstable manifolds are absolutely continuous with
respect to Riemannian volume on these manifolds.  The SRB measure $\nu$
satisfies the central limit theorem and exhibits exponential decay of correlations for
H\"{o}lder continuous observables.  For Lebesgue almost every $\mbi{x}$ in the basin of
attraction $U$, the orbit of $\mbi{x}$ has a positive Lyapunov exponent and is
asymptotically distributed according to $\nu$ in the sense that for every continuous
function $\vp : U \to \mbb{R}$, we have
\begin{equation}\label{e:time_space}
\lim_{m \to \infty} \frac{1}{m} \sum_{i=0}^{m-1} \vp (G_{T}^{i} (\mbi{x})) = \int \vp \, d
\nu.
\end{equation}
Notice that this statement is substantially stronger than the conclusion of the Birkhoff
ergodic theorem.  The Birkhoff ergodic theorem implies that~\eqref{e:time_space} holds for
$\nu$ almost every $\mbi{x}$.  However, $\nu$ is singular with respect to Lebesgue measure
(supported on a set of Lebesgue measure zero) because the dynamics are dissipative.  We
prove that~\eqref{e:time_space} holds for Lebesgue almost every $\mbi{x} \in U$.
See~\bpref{li:sa1}--\bpref{li:sa4} in Section~\ref{s:r1t} for a more precise description
of the dynamical properties of $G_{T}$ for $T \in \De$.

We now define the shear integral.  In Section~\ref{s:singular_lim} we derive a normal form
of~\eqref{e:intrinsic} that is valid in $\tilde{M}$.  The normal form, expressed in the
natural $(s, \mbi{z})$-coordinates introduced in Section~\ref{ss:nf_der}, is given by
\begin{align}
\sublabon{equation}
\tdf{t}{s} &= \norm{\mbi{f} (\bsym{\ga} (s))}^{-1} + \inproa{\bsym{\be} (s)}{\mbi{z}}
+ \om_{1} (s, \mbi{z})
\label{e:nf_intrin_t}\\
\tdf{\mbi{z}}{s} &= \msf{A} \mbi{z} + \bsym{\om}_{2} (s, \mbi{z})
\label{e:nf_intrin_z}
\end{align}
\sublaboff{equation}%
Here $\inproa{\cdot}{\cdot}$ denotes the inner product on $\mbb{R}^{n-1}$.  Functions
depending on $s$ in~\eqref{e:nf_intrin_t}--\eqref{e:nf_intrin_z} are periodic in $s$ with
period $2L$.  The matrix $\msf{A}$ is in Jordan canonical form.  The functions $\om_{1}$
and $\bsym{\om}_{2}$ represent higher order corrections.  The function $\bsym{\be}$ gives
the pointwise magnitude and direction of the shear.  Define the {\bfseries \itshape shear
  integral} $\bsym{\Si}$ by
\begin{equation*}
\bsym{\Si} = (\Si_{1}, \ldots, \Si_{n-1}) \defas \int_{0}^{2L} \bsym{\be} (\ta) \, d \ta
\end{equation*}
and define the {\bfseries \itshape shear factor} $\si$ by $\si \defas \norm{\bsym{\Si}}$.

Having defined the shear integral, we describe the setting of the main theorem.  We
identify intrinsic parameters (parameters associated with $\mbi{f}$) and external
parameters (parameters associated with the external forcing).  We fix the normalized shear
vector $\frac{\bsym{\Si}}{\si}$ and view the shear factor $\si$ as the first intrinsic
parameter.  The second intrinsic parameter quantifies the strength of the contraction near
the limit cycle and is derived from $\msf{A}$.  We assume for the sake of simplicity that
$\msf{A}$ is a diagonal matrix given by $\msf{A} = \diag (\la_{1}, \ldots, \la_{n-1})$
where $0 > \la_{1} \geqs \la_{2} \geqs \cdots \geqs \la_{n-1}$ are the eigenvalues of
$\msf{A}$.  We fix the eigenvalue ratios $\mu_{i} = \frac{\la_{1}}{\la_{i}}$ for $1 \leqs
i \leqs n-1$ and we view the weakest eigenvalue $\la_{1}$ as an intrinsic parameter.  The
only external parameter is $\ve$, the factor that controls the amplitude of the external
forcing.  We fix $\rh > 0$.  A key parameter derived from $\ve$, $\si$, and $\la_{1}$ is
the {\bfseries \itshape hyperbolicity factor} $\frac{\ve \si}{|\la_{1}|}$.

One additional ingredient is needed.  Even if $\si$ is large and $|\la_{1}|$ is small, a
strange attractor cannot emerge unless the forcing $\mbi{F}$ acts in direction(s) in which
shear is present.  We express this idea by introducing a certain function on the circle
$\mbb{S} \defas \frac{\mbb{R}}{2L \mbb{Z}}$.  We identify $\mbb{S}$ with the interval
$[0,2L)$.  In Section~\ref{s:singular_lim} we derive a normal form of the forced
system~\eqref{e:forced} that is valid in $\tilde{M}$ when the forcing is active ($P_{\rh,
  T} = 1$):
\begin{align}
\sublabon{equation}
\tdf{t}{s} &= \norm{\mbi{f} (\bsym{\ga} (s))}^{-1} + \inproa{\bsym{\be} (s)}{\mbi{z}}
+ \om_{3} (s, \mbi{z})
\label{e:nf_force_t}\\
\tdf{\mbi{z}}{s} &= \msf{A} \mbi{z} + \ve \bsym{\ze} (s) + \bsym{\om}_{4} (s, \mbi{z})
\label{e:nf_force_z}
\end{align}
\sublaboff{equation}%
Functions depending on $s$ in~\eqref{e:nf_force_t}--\eqref{e:nf_force_z} are periodic in
$s$ of period $2L$.  The functions $\om_{3}$ and $\bsym{\om}_{4}$ are higher order
corrections.  The function $\bsym{\ze}$ is related to the projection of $\mbi{F}$ in
directions orthogonal to $\Ga$.  For $s_{0} \in \mbb{S}$, define $\tilde{s}$ implicitly by
\begin{equation*}
\rh = \int_{s_{0}}^{\tilde{s}} \norm{\mbi{f} (\bsym{\ga} (\ta))}^{-1} \, d \ta.
\end{equation*}
Define the vector
\begin{equation*}
\mbi{d} \defas \left( \frac{\Si_{i} \mu_{i}}{\si} \right)_{i=1}^{n-1}
\end{equation*}
and define $\Phi : \mbb{S} \to \mbb{R}$ by
\begin{equation}\label{e:morse_ord1}
\Phi (s_{0}) = \left\langle \mbi{d}, \int_{s_{0}}^{\tilde{s}} \bsym{\ze} (\ta) \, d \ta \right\rangle.
\end{equation}
We say that $\Phi$ is a {\bfseries \itshape Morse function} if the critical set $C(\Phi) =
\{ s \in \mbb{S} : \Phi'(s) = 0 \}$ is finite and if for every $s \in C(\Phi)$, we have
$\Phi''(s) \neq 0$.  We are now in position to state the main theorem.  In
Theorem~\ref{mt:main_1}, we assume that the radius of $M$ is $\leqs \ka_{0} \ve$ for some
constant $\ka_{0} > 0$.

\begin{m_theorem}\label{mt:main_1}
  Let $G_{T}$ denote the time-$T$ map associated with~\eqref{e:forced}.  Suppose that the
  function $\Phi$ defined by~\eqref{e:morse_ord1} is a Morse function.  Then there exist a
  small constant $\ka_{1} > 0$ and a large constant $\ka_{2} > \ka_{1}$ such that the
  following holds.  If
\begin{enumerate}
\item
$|\la_{1}| < \ka_{1}$,
\label{hyp:contract}
\item
$\frac{\ve}{|\la_{1}|} < \ka_{1}$,
\item
$\frac{\ve \si}{|\la_{1}|} > \ka_{2}$,
\end{enumerate}
then there exists $T_{1} > 0$ and a set $\De \subset [T_{1}, \infty)$ of positive Lebesgue
measure such that for $T \in \De$, $G_{T}$ admits a strange attractor $\Om$ in $M$ and
satisfies~\bpref{li:sa1}--\bpref{li:sa4} from Section~\ref{s:r1t}.  For every interval $I
\subset [T_{1}, \infty)$ of length $1$, $\ell (\De \cap I) > 0$, where $\ell$ denotes the
Lebesgue measure on $\mbb{R}$.
\end{m_theorem}

\begin{remark}
  The assumption that $\Phi$ is a Morse function is quite mild and should hold for a
  typical forcing vector field $\mbi{F}$.  We do not formulate precise results of this
  type in this paper, but such results should hold in terms of both topological genericity
  and prevalence.  Prevalence is a measure-theoretic notion of genericity that generalizes
  the concept of `Lebesgue almost every' to infinite-dimensional spaces.  It provides a
  powerful framework for describing generic phenomena in a probabilistic way (see
  {\itshape e.g.}~\cite{HbStYj1992, HbStYj1993, OwYj2005}).
\end{remark}

\begin{remark}
  Theorem~\ref{mt:main_1} concludes that $G_{T}$ exhibits sustained, observable chaos for
  a set of values of $T$ of positive Lebesgue measure rather than for all $T \in [T_{1},
  \infty)$.  This is not a consequence of the nature of the proof.  Rather, it is a
  fundamental consequence of the fact that an alternate scenario competes with the SRB
  scenario in the space of $T$-values.  For an open set $\mscr{S}$ of $T$-values in
  $[T_{1}, \infty)$, the basin $U$ contains a $G_{T}$-invariant Cantor set on which
  $G_{T}$ is uniformly hyperbolic (a horseshoe) and a periodic sink.  The trajectory of
  Lebesgue almost every $\mbi{x} \in U$ converges to the periodic sink.  Thus for $T \in
  \mscr{S}$, $G_{T}$ exhibits transient chaos: a typical trajectory in the basin will move
  erratically for some time due to the presence of the horseshoe before finally converging
  to the periodic sink.
\end{remark}

\begin{remark}
The function $\Phi$ does not depend on the parameters $\la_{1}$, $\si$, and $\ve$.  
\end{remark}

Theorem~\ref{mt:main_1} is related to $2$ results obtained by Wang and Young
in~\cite{WqYls2003}.  Wang and Young consider limit cycles forced by periodic
$\de$-function kicks.  First, they prove that any limit cycle, when suitably kicked, can
be transformed into a strange attractor.  This result is universal but not constructive.
An artificially-strong kick is needed if geometric conditions are unfavorable for the
creation of nonuniform hyperbolicity.  Second, they prove that the Hopf limit cycle that
emerges from a supercritical Hopf bifurcation can be transformed into a strange
attractor.  Here the so-called twist factor plays the role of the shear integral.  Unlike
the shear integral, the twist factor is local in the sense that it depends only on
derivatives of the vector field at the bifurcation parameter.

Many of the quantities in Theorem~\ref{mt:main_1} are required to be sufficiently large or
sufficiently small.  This is an unavoidable consequence of the perturbative nature of the
analytic techniques used in the proof.  However, numerical evidence suggests that
shear-induced chaos emerges over parameter ranges that far exceed those to which the
rigorous analysis applies.  For example, Lin and Young~\cite{LkYls2008} conduct numerical
studies of a linear shear flow model previously studied by Zaslavsky~\cite{Zg1978}.  The
work of Lin and Young also provides numerical evidence that the temporal form of the kicks
need not be periodic: temporally-sustained chaotic behavior is observed for random kicks
at Poisson-distributed times and for continuous-time forcing by white noise.


\section{Derivation of the singular limit}\label{s:singular_lim}

\subsection{Derivation of the normal forms}\label{ss:nf_der}

We derive the normal forms~\eqref{e:nf_intrin_t}--\eqref{e:nf_intrin_z}
and~\eqref{e:nf_force_t}--\eqref{e:nf_force_z} that are valid in a small neighborhood of
$\Ga$.  For $s \in \mbb{S}$, let $\{ \mbi{e}_{i} (s) \}_{i=1}^{n}$ be an orthonormal basis
for $\mbb{R}^{n}$ such that $\mbi{e}_{n} (s) = \bsym{\ga}'(s)$ (where $\bsym{\ga}'$
denotes the derivative of $\bsym{\ga}$ with respect to $s$) and $\mbi{e}_{i}$ is a $C^{5}$
function of $s$ for all $1 \leqs i \leqs n$.  One may choose the first $n-1$ vectors in
many ways.  For example, if $\bsym{\ga}$ is at least $C^{n+5}$ and the first $n$ derivatives
of $\bsym{\ga}$ are linearly independent, then one may construct the basis by applying the
Gram-Schmidt procedure to the first $n$ derivatives of $\bsym{\ga}$.  For any $\mbi{x} \in
\mbb{R}^{n}$ sufficiently close to $\Ga$, there exist unique $s \in \mbb{S}$ and $\mbi{y}
= (y_{1}, \ldots, y_{n-1})$ such that
\begin{equation}\label{e:local_rep}
\mbi{x} = \bsym{\ga} (s) + \sum_{i=1}^{n-1} y_{i} \mbi{e}_{i} (s).
\end{equation}
We use $(s, \mbi{y})$ as new phase variables.

Define
\begin{equation*}
\msf{E} (s) =
\begin{pmatrix}
(\mbi{e}_{1} (s))^{\msf{T}}\\
(\mbi{e}_{2} (s))^{\msf{T}}\\
\vdots\\
(\mbi{e}_{n} (s))^{\msf{T}}
\end{pmatrix}
\end{equation*}
Differentiating $\msf{E} (s)$ with respect to $s$, we have $\msf{E}'(s) = \msf{K} (s)
\msf{E} (s)$ where $\msf{K} (s) = (k_{j,i}(s))$ is a skew-symmetric matrix of generalized
curvatures defined by $k_{j,i}(s) = \inproa{\mbi{e}_{j}'(s)}{\mbi{e}_{i}(s)}$.  If the
first $n$ derivatives of $\bsym{\ga}$ are used to create $\msf{E}$, then this differential
equation is the classical Frenet-Serret equation from differential geometry.  For $1 \leqs
i \leqs n$, define the vector
\begin{equation*}
\mbi{k}_{i} (s) =
\begin{pmatrix}
k_{1,i} (s)\\
k_{2,i} (s)\\
\vdots\\
k_{n-1,i} (s)
\end{pmatrix}
\end{equation*}
Differentiating~\eqref{e:local_rep} with respect to $t$, we obtain
\begin{equation}\label{e:dxdt_tform}
\tdf{\mbi{x}}{t} = \sum_{i=1}^{n-1} \tdf{y_{i}}{t} \mbi{e}_{i} (s) + \tdf{s}{t} \left(
  \bsym{\ga}'(s) + \sum_{j=1}^{n-1} y_{j} \mbi{e}_{j}'(s) \right) = \mbi{f} (\mbi{x}) +
\ve P_{\rh,T} (t) \mbi{F} (\mbi{x}).
\end{equation}
Taking the inner product of~\eqref{e:dxdt_tform} with respect to $\mbi{e}_{i} (s)$ for $1
\leqs i \leqs n-1$ yields
\begin{equation*}
\tdf{y_{i}}{t} = \inproa{\mbi{f} (\mbi{x})}{\mbi{e}_{i} (s)} + \ve P_{\rh,T} (t)
\inproa{\mbi{F} (\mbi{x})}{\mbi{e}_{i} (s)} - \tdf{s}{t} \inproa{\mbi{y}}{\mbi{k}_{i}
  (s)}.
\end{equation*}
Taking the inner product of~\eqref{e:dxdt_tform} with respect to $\mbi{e}_{n} (s)$ yields
\begin{equation*}
\tdf{s}{t} (\inproa{\mbi{y}}{\mbi{k}_{n} (s)} + 1) = \inproa{\mbi{f}
  (\mbi{x})}{\mbi{e}_{n} (s)} + \ve P_{\rh,T} (t) \inproa{\mbi{F} (\mbi{x})}{\mbi{e}_{n}
  (s)}.
\end{equation*}
Notice that $\inproa{\mbi{y}}{\mbi{k}_{n} (s)} + 1 \neq 0$ if $\| \mbi{y} \|$ is
sufficiently small.  Consequently, the system
\begin{align}
\sublabon{equation}
\tdf{s}{t} &= \frac{1}{1 + \inproa{\mbi{y}}{\mbi{k}_{n} (s)}} (\inproa{\mbi{f}
  (\mbi{x})}{\mbi{e}_{n} (s)} + \ve P_{\rh,T} (t) \inproa{\mbi{F} (\mbi{x})}{\mbi{e}_{n}
  (s)})
\label{e:nf_prelim_s}\\
\tdf{y_{i}}{t} &= \inproa{\mbi{f} (\mbi{x})}{\mbi{e}_{i} (s)} + \ve P_{\rh,T} (t)
\inproa{\mbi{F} (\mbi{x})}{\mbi{e}_{i} (s)} - \tdf{s}{t} \inproa{\mbi{y}}{\mbi{k}_{i} (s)}
\label{e:nf_prelim_yi}
\end{align}
\sublaboff{equation}%
is valid in a small neighborhood of $\Ga$.

We now extract the terms of leading order in~\eqref{e:nf_prelim_s}
and~\eqref{e:nf_prelim_yi}.  For $1 \leqs j \leqs n$, define $\psi_{j} (s, \mbi{y}) =
\inproa{\mbi{f} (\mbi{x})}{\mbi{e}_{j} (s)}$.  For $1 \leqs i \leqs n-1$, we have
\begin{equation*}
\psi_{i} (s, \mbi{y}) = \inproa{\subpsup{\bsym{\psi}}{i}{1} (s)}{\mbi{y}} + \mcal{O}_{s, \mbi{y}}
(\norm{\mbi{y}}^{2})
\end{equation*}
where
\begin{equation*}
\subpsup{\bsym{\psi}}{i}{1} (s) = \left. \pdf{\inproa{\mbi{f}
      (\mbi{x})}{\mbi{e}_{i} (s)}}{\mbi{y}} \right|_{\mbi{y} = \mbold{0}}.
\end{equation*}
Here $\mcal{O}_{s,\mbi{y}} (\norm{\mbi{y}}^{2})$ denotes a function of $s$ and $\mbi{y}$
for which there exists a constant $K > 0$ independent of $s$ and $\mbi{y}$ such that
$|\mcal{O}_{s,\mbi{y}} (\norm{\mbi{y}}^{2})| \leqs K \norm{\mbi{y}}^{2}$.  Expanding
$\psi_{n} (s, \mbi{y})$, we have
\begin{equation*}
\psi_{n} (s, \mbi{y}) = \subpsup{\psi}{n}{0} (s) + \inproa{\subpsup{\bsym{\psi}}{n}{1}
  (s)}{\mbi{y}} + \mcal{O}_{s, \mbi{y}} (\norm{\mbi{y}}^{2})
\end{equation*}
where
\begin{align*}
\subpsup{\psi}{n}{0} (s) &= \norm{\mbi{f} (\bsym{\ga} (s))}\\
\subpsup{\bsym{\psi}}{n}{1} (s) &= \left. \pdf{\inproa{\mbi{f} (\mbi{x})}{\mbi{e}_{n}
      (s)}}{\mbi{y}} \right|_{\mbi{y} = \mbold{0}}.
\end{align*}
Set $\phi_{j} (s, \mbi{y}) = \inproa{\mbi{F} (\mbi{x})}{\mbi{e}_{j} (s)}$ for $1 \leqs j
\leqs n$.  Writing~\eqref{e:nf_prelim_s} and~\eqref{e:nf_prelim_yi} in terms of $\psi_{j}$
and $\phi_{j}$, when the forcing is active ($P_{\rh,T} (t)=1$) we obtain
\begin{equation}\label{e:nf_expand}
\left\{
\begin{aligned}
\tdf{t}{s} &= \frac{1}{\subpsup{\psi}{n}{0} (s) + \ve \phi_{n} (s, \mbi{y})} \left( 1 +
  \left[ \mbi{k}_{n} (s) - \frac{\subpsup{\bsym{\psi}}{n}{1} (s)}{\subpsup{\psi}{n}{0} (s)
      + \ve \phi_{n} (s, \mbi{y})} \right] \cdot \mbi{y} + \mcal{O}_{s, \mbi{y}}
  (\norm{\mbi{y}}^{2}) \right)\\
\tdf{y_{i}}{s} &= \frac{\ve \phi_{i} (s, \mbi{y})}{\subpsup{\psi}{n}{0} (s) + \ve \phi_{n}
  (s, \mbi{y})} + \left( \frac{\subpsup{\bsym{\psi}}{i}{1} (s)}{\subpsup{\psi}{n}{0} (s) +
    \ve \phi_{n} (s, \mbi{y})} - \mbi{k}_{i} (s) \right) \cdot \mbi{y}\\
&\quad {}+ \left( \frac{\ve \phi_{i} (s, \mbi{y})}{\subpsup{\psi}{n}{0} (s) + \ve \phi_{n}
  (s, \mbi{y})} \right) \left[ \mbi{k}_{n} (s) - \frac{\subpsup{\bsym{\psi}}{n}{1} (s)}{\subpsup{\psi}{n}{0} (s)
      + \ve \phi_{n} (s, \mbi{y})} \right] \cdot \mbi{y} + \mcal{O}_{s, \mbi{y}}
  (\norm{\mbi{y}}^{2})
\end{aligned}
\right.
\end{equation}
When the forcing is off ($P_{\rh,T} (t) = 0$), we have
\begin{equation}\label{e:nfrelax1}
\left\{
\begin{aligned}
\tdf{t}{s} &= \frac{1}{\subpsup{\psi}{n}{0} (s)} \left( 1 + \left[ \mbi{k}_{n} (s) -
    \frac{\subpsup{\bsym{\psi}}{n}{1} (s)}{\subpsup{\psi}{n}{0} (s)} \right] \cdot \mbi{y}
  + \mcal{O}_{s, \mbi{y}} (\norm{\mbi{y}}^{2}) \right)\\
\tdf{y_{i}}{s} &= \left( \frac{\subpsup{\bsym{\psi}}{i}{1} (s)}{\subpsup{\psi}{n}{0} (s)}
  - \mbi{k}_{i} (s) \right) \cdot \mbi{y} + \mcal{O}_{s, \mbi{y}} (\norm{\mbi{y}}^{2})
\end{aligned}
\right.
\end{equation}
Define
\begin{align*}
b_{0} (s) &\defas \frac{1}{\subpsup{\psi}{n}{0} (s)}\\
\mbi{b}_{1} (s) &\defas \frac{1}{\subpsup{\psi}{n}{0} (s)} \left( \mbi{k}_{n} (s) -
  \frac{\subpsup{\bsym{\psi}}{n}{1} (s)}{\subpsup{\psi}{n}{0} (s)} \right)
\end{align*}
and let $\tilde{\msf{A}} (s)$ denote the $(n-1) \times (n-1)$ matrix with $i^{\text{th}}$ row
given by
\begin{equation*}
\left( \frac{\subpsup{\bsym{\psi}}{i}{1} (s)}{\subpsup{\psi}{n}{0} (s)} - \mbi{k}_{i} (s) \right)^{\msf{T}}.
\end{equation*}
In terms of $b_{0}$, $\mbi{b}_{1}$, and $\tilde{\msf{A}}$, system~\eqref{e:nfrelax1}
becomes
\begin{equation}\label{e:nfrelax2}
\left\{
\begin{aligned}
\tdf{t}{s} &= b_{0} (s) + \inproa{\mbi{b}_{1} (s)}{\mbi{y}} + \mcal{O}_{s, \mbi{y}}
(\norm{\mbi{y}}^{2})\\
\tdf{\mbi{y}}{s} &= \tilde{\msf{A}} (s) \mbi{y} + \mcal{O}_{s, \mbi{y}}
(\norm{\mbi{y}}^{2})
\end{aligned}
\right.
\end{equation}
Applying the Floquet theorem, there exists a real-valued, periodic $(n-1) \times (n-1)$ matrix $\msf{P}
(s)$ of period $2L$ such that setting $\mbi{z} = \msf{P}^{-1} (s) \mbi{y}$, we
transform~\eqref{e:nfrelax2} into
\begin{align}
\sublabon{equation}
\tdf{t}{s} &= b_{0} (s) + ((\mbi{b}_{1} (s))^{\msf{T}} \msf{P} (s)) \mbi{z} + h_{2} (s,
\mbi{z})
\label{e:nf_relax_t}\\
\tdf{\mbi{z}}{s} &= \msf{A} \mbi{z} + \mbi{h}_{1} (s, \mbi{z})
\label{e:nf_relax_z}
\end{align}
\sublaboff{equation}%
This is the normal form of~\eqref{e:forced} on which we will base our analysis of the flow
during the relaxation period (when $P_{\rh,T} (t) = 0$).  We obtain the normal form
of~\eqref{e:forced} during the forcing period (when $P_{\rh,T} (t) = 1$) by
writing~\eqref{e:nf_expand} in $(s, \mbi{z})$-coordinates, giving
\begin{align}
\sublabon{equation}
\tdf{t}{s} &= b_{0} (s) + ((\mbi{b}_{1} (s))^{\msf{T}} \msf{P} (s)) \mbi{z} +
\mcal{O}_{s, \mbi{z}} (\ve) + \mcal{O}_{s, \mbi{z}} (\ve \mbi{z}) + \mcal{O}_{s, \mbi{z}}
(\norm{\mbi{z}}^{2})
\label{e:nf_full_t}\\
\tdf{\mbi{z}}{s} &= \msf{A} \mbi{z} + \frac{\ve \msf{P}^{-1} (s) \bsym{\phi}
  (s, \mbi{0})}{\subpsup{\psi}{n}{0} (s)} + \mcal{O}_{s, \mbi{z}}
(\ve \mbi{z}) + \mcal{O}_{s, \mbi{z}} (\ve^{2}) + \mcal{O}_{s, \mbi{z}} (\norm{\mbi{z}}^{2})
\label{e:nf_full_z}
\end{align}
\sublaboff{equation}%
where $\bsym{\phi} (s, \mbi{0}) = (\phi_{1} (s, \mbi{0}), \ldots, \phi_{n-1}
(s, \mbi{0}))^{\msf{T}}$.

\subsection{A general form of the singular limit}\label{ss:general_sl}

Let $\tilde{M} \approx \Ga \times D$ be a tubular neighborhood of $\Ga$ in $\mbb{R}^{n}$,
where $D$ is a disk of sufficiently small radius so that the normal
form~\eqref{e:nf_full_t}--\eqref{e:nf_full_z} is valid.  Let $M \approx \Ga \times
\frac{1}{2} D$.  We define flow-induced maps $H_{k} : M \to \tilde{M}$ and $H_{r} :
\tilde{M} \to \tilde{M}$ as follows.  Let $H_{k}$ be the time-$\rh$ map associated with
the forced system~\eqref{e:nf_full_t}--\eqref{e:nf_full_z}.  We call $H_{k}$ the
\textquoteleft kick'.  Notice that for $\ve$ sufficiently small, $H_{k}$ maps $M$ into
$\tilde{M}$.  Let $H_{r}$ be the time-$(T - \rh)$ map associated with the relaxation
system~\eqref{e:nf_relax_t}--\eqref{e:nf_relax_z}.  We call $H_{r}$ the relaxation map.
There exists $T_{0} = T_{0} (\ve)$ such that if $T \geqs T_{0}$, then $H_{r}$ maps
$\tilde{M}$ into $\interior (M)$.  The composition $G_{T} \defas H_{r} \circ H_{k}$ is the
time-$T$ map generated by the flow.  Our goal is to show that the family $\{ G_{T} : M \to
\interior (M), \; \: T \geqs T_{0} \}$ of diffeomorphisms on $M$ has a well-defined
singular limit in a certain sense as $T \to \infty$.

Let $(s_{0}, \mbi{y}_{0}) \in M$.  We write $H_{k} (s_{0}, \mbi{y}_{0}) = (\hat{s},
\hat{\mbi{z}})$ and compute $H_{r} (\hat{s}, \hat{\mbi{z}})$.
Integrating~\eqref{e:nf_relax_z}, we have
\begin{equation*}
\mbi{z} (s) = e^{(s - \hat{s}) \msf{A}} \left( \hat{\mbi{z}} + \int_{\hat{s}}^{s} e^{-(\ta
    - \hat{s}) \msf{A}} \mbi{h}_{1} (\ta, \mbi{z} (\ta)) \, d \ta \right).
\end{equation*}
Integrating~\eqref{e:nf_relax_t}, we have
\begin{equation}
\label{e:nf_rel_t_int}
T - \rh = \int_{\hat{s}}^{s(T)} b_{0} (\ta) \, d \ta + \hat{\mbi{z}} \cdot \int_{\hat{s}}^{s(T)}
\mbi{b}_{1} (\ta)^{\msf{T}} \msf{P} (\ta) e^{(\ta - \hat{s}) \msf{A}} \, d \ta +
\sum_{k=1}^{2} E_{k} (s(T)),
\end{equation}
where the error terms are given by
\begin{align*}
E_{1} (s(T)) &= \int_{\hat{s}}^{s(T)} \mbi{b}_{1} (\ta)^{\msf{T}} \msf{P} (\ta) e^{(\ta -
  \hat{s}) \msf{A}} \int_{\hat{s}}^{\ta} e^{-(\xi - \hat{s}) \msf{A}} \mbi{h}_{1} (\xi,
\mbi{z} (\xi)) \, d \xi \, d \ta\\
E_{2} (s(T)) &= \int_{\hat{s}}^{s(T)} h_{2} (\ta, \mbi{z} (\ta)) \, d \ta.
\end{align*}
Letting $T \to \infty$ in~\eqref{e:nf_rel_t_int} yields nothing meaningful.  However, we
use the fact that $s$ is can be computed modulo $2L$ to introduce an auxiliary parameter
$a \in \mbb{S}$ and thereby obtain the singular limit.  Recall that $p_{0}$ is the period
of $\bsym{\et}$.  As $a$ varies from $0$ to $2L$, $\bsym{\ga}$ traverses $\Ga$ $2$ times.
Let $\hat{t} : [0, 2L) \to [0, 2p_{0})$ be the strictly increasing function defined by
$\bsym{\eta} (\hat{t} (a)) = \bsym{\ga} (a)$.  For $m \in \mbb{Z}^{+}$ and $a \in
\mbb{S}$, set $T = \rh + 2p_{0}m + \hat{t} (a)$.  Substituting
into~\eqref{e:nf_rel_t_int}, writing $s(\rh + 2p_{0}m + \hat{t} (a)) = \hat{s} + 2Lm +
\tilde{s} (\rh + 2p_{0}m + \hat{t} (a))$, and using the fact that
\begin{equation*}
\int_{v}^{v+2Lm} b_{0} (\ta) \, d \ta = 2p_{0}m
\end{equation*}
for all $v \in \mbb{R}$, we obtain
\begin{equation}
\label{e:s_tilde}
\begin{aligned}
\hat{t} (a) &= \int_{\hat{s}}^{\hat{s} + \tilde{s} (\rh + 2p_{0}m + \hat{t} (a))} b_{0}
(\ta) \, d \ta + \hat{\mbi{z}} \cdot \int_{\hat{s}}^{s (\rh + 2p_{0}m + \hat{t} (a))}
\mbi{b}_{1} (\ta)^{\msf{T}} \msf{P} (\ta) e^{(\ta - \hat{s}) \msf{A}} \, d \ta \\ 
&\qquad {}+
\sum_{k=1}^{2} E_{k} (s(\rh + 2p_{0}m + \hat{t} (a)).
\end{aligned}
\end{equation}
Define $G_{a, m^{-1}} : M \to \interior (M)$ by $G_{a, m^{-1}} (s_{0}, \mbi{y}_{0}) = (s(\rh + 2p_{0}m
+ \hat{t} (a)), \mbi{y} (\rh + 2p_{0}m + \hat{t} (a))$.  It follows
from~\cite[Proposition~3.1]{WqYls2003} that there exists $s_{\infty} (s_{0},
\mbi{y}_{0}, a)$ such that
\begin{equation*}
\lim_{m \to \infty} \hat{s} + \tilde{s} (\rh + 2p_{0}m + \hat{t} (a)) = s_{\infty} (s_{0},
\mbi{y}_{0}, a)
\end{equation*}
and $s_{\infty} (s_{0}, \mbi{y}_{0}, a)$ is defined implicitly by taking the $m \to
\infty$ limit in~\eqref{e:s_tilde}:
\begin{equation}
\label{e:s_infty_1}
\hat{t} (a) = \int_{\hat{s}}^{s_{\infty} (s_{0}, \mbi{y}_{0}, a)} b_{0} (\ta) \, d \ta +
\inproav{\hat{\mbi{z}}}{\int_{\hat{s}}^{\infty} \mbi{b}_{1} (\ta)^{\msf{T}} \msf{P} (\ta)
e^{(\ta - \hat{s}) \msf{A}} \, d \ta} + \sum_{k=1}^{2} E_{k} (\infty).
\end{equation}
The family of maps $\{ G_{a, 0} : M \to \Ga \times \{ \mbi{0} \} \}_{a \in \mbb{S}}$
defined by
\begin{equation*}
G_{a, 0} (s_{0}, \mbi{y}_{0}) = (s_{\infty} (s_{0}, \mbi{y}_{0}, a), \mbi{0})
\end{equation*}
is the desired singular limit.  It follows from~\cite[Proposition 3.1]{WqYls2003} that the
maps
\begin{equation*}
(s_{0}, \mbi{y}_{0}, a) \mapsto G_{a,m^{-1}} (s_{0}, \mbi{y}_{0})
\end{equation*}
converge to the map
\begin{equation*}
(s_{0}, \mbi{y}_{0}, a) \mapsto G_{a,0} (s_{0}, \mbi{y}_{0})
\end{equation*}
in $C^{3} (M \times \mbb{S})$ as $m \to \infty$.

\subsection{A computable form of the singular limit}\label{ss:com_sl}

>From this point forward, we assume the setting of Theorem~\ref{mt:main_1}.  We now extract
the primary terms in the right side of~\eqref{e:s_infty_1}.  Recall that the shear integral
$\bsym{\Si}$ is defined by
\begin{equation*}
\bsym{\Si} = (\Si_{1}, \ldots, \Si_{n-1}) = \int_{0}^{2L} \mbi{b}_{1} (\ta)^{\msf{T}} \msf{P} (\ta) \, d \ta.
\end{equation*}
and that the shear factor is given by $\si = \| \bsym{\Si} \|$.  We assume that the operator
$\msf{A}$ is diagonalizable and that the $\mbi{z}$-coordinate has been chosen such that
$\msf{A} = \diag (\la_{1}, \ldots, \la_{n-1})$, where $0 > \la_{1} \geqs \la_{2} \geqs
\cdots \geqs \la_{n-1}$ are the eigenvalues of $\msf{A}$.  Fix the normalized shear vector
$\frac{\bsym{\Si}}{\si}$ and the eigenvalue ratios $\mu_{i} =
\frac{\la_{1}}{\la_{i}}$ for $1 \leqs i \leqs n-1$.  Set $\rh = 1$ for notational
simplicity.  We regard $\si$, $\ve$, and $\la_{1}$ as the parameters associated with the
singular limit.

Expanding the second term on the right side of~\eqref{e:s_infty_1}, we have
\begin{equation}\label{e:mean_ext}
\begin{aligned}
\int_{\hat{s}}^{\infty} \mbi{b}_{1} (\ta)^{\msf{T}} \msf{P} (\ta) e^{(\ta - \hat{s})
  \msf{A}} \, d \ta &= \int_{\hat{s}}^{\infty} \bsym{\Si} e^{(\ta - \hat{s}) \msf{A}} \, d
\ta + \int_{\hat{s}}^{\infty} (\mbi{b}_{1} (\ta)^{\msf{T}} \msf{P} (\ta) - \bsym{\Si})
e^{(\ta - \hat{s}) \msf{A}} \, d \ta\\
&= \bar{\mbi{d}} + \int_{\hat{s}}^{\infty} (\mbi{b}_{1} (\ta)^{\msf{T}} \msf{P} (\ta) - \bsym{\Si})
e^{(\ta - \hat{s}) \msf{A}} \, d \ta,
\end{aligned}
\end{equation}
where
\begin{equation*}
\bar{\mbi{d}} = \int_{\hat{s}}^{\infty} \bsym{\Si} e^{(\ta - \hat{s}) \msf{A}} \, d \ta =
\left( - \frac{\Si_{i}}{\la_{i}} \right)_{i=1}^{n-1}.
\end{equation*}
Let $\tilde{H}_{k} : M \to \tilde{M}$ be the time-$1$ map generated by the system
\begin{align}
\sublabon{equation}
\tdf{t}{s} &= b_{0} (s)
\label{e:nf_full_t_l}\\
\tdf{\mbi{z}}{s} &= \frac{\ve \msf{P}^{-1} (s) \bsym{\phi} (s,\mbi{0})}{\subpsup{\psi}{n}{0} (s)}
\label{e:nf_full_z_l}
\end{align}
\sublaboff{equation}%
obtained from~\eqref{e:nf_full_t}--\eqref{e:nf_full_z} by retaining only the terms of
leading order.  For $(s_{0}, \mbi{y}_{0}) \in M$, write $\tilde{H}_{k} (s_{0},
\mbi{y}_{0}) = (\tilde{s}, \tilde{\mbi{z}})$.  Integrating~\eqref{e:nf_full_t_l}
and~\eqref{e:nf_full_z_l} gives
\begin{equation}\label{e:int_nf_l}
\begin{aligned}
1 &= \int_{s_{0}}^{\tilde{s}} b_{0} (\ta) \, d \ta,\\
\tilde{\mbi{z}} &= \mbi{z}_{0} + \ve \int_{s_{0}}^{\tilde{s}} \frac{\msf{P}^{-1} (\ta)
  \bsym{\phi} (\ta,\mbi{0})}{\subpsup{\psi}{n}{0} (\ta)} \, d \ta.
\end{aligned}
\end{equation}

\begin{proposition}\label{p:f_per_dif}
There exists a system constant $K_{0} > 0$ such that
\begin{equation}\label{e:f_per_dif}
\hat{s} = \tilde{s} + \xi_{1} (s_{0}, \mbi{y}_{0}), \quad \hat{\mbi{z}} = \tilde{\mbi{z}}
+ \bsym{\xi}_{2} (s_{0}, \mbi{y}_{0})
\end{equation}
where
\begin{equation*}
\norm{\xi_{1} | \{ \mbi{y}_{0} = \mbi{0} \}}_{C^{3} (\mbb{S})} \leqs K_{0} \ve, \quad
\norm{\bsym{\xi}_{2} | \{ \mbi{y}_{0} = \mbi{0} \}}_{C^{3} (\mbb{S})} \leqs K_{0} \ve
|\la_{1}|.
\end{equation*} 
\end{proposition}

Setting $\mbi{y}_{0} = \mbi{0}$, define $g (s_{0}, a) = s_{\infty} (s_{0}, \mbi{0}, a)$.
Substituting~\eqref{e:mean_ext}, \eqref{e:int_nf_l}, and~\eqref{e:f_per_dif} into~\eqref{e:s_infty_1}, the
value $g (s_{0}, a)$ is defined implicitly by
\begin{equation}\label{e:slform_pre}
\begin{split}
\hat{t} (a) + 1 &= \int_{s_{0}}^{g (s_{0}, a)} b_{0} (\ta) \, d \ta +
\inproa{(\tilde{\mbi{z}} + \bsym{\xi}_{2} (s_{0}, \mbi{0}))}{\bar{\mbi{d}}}\\
&\quad {}-
\int_{\tilde{s}}^{\hat{s}} b_{0} (\ta) \, d \ta + \hat{\mbi{z}} \cdot
\int_{\hat{s}}^{\infty} (\mbi{b}_{1} (\ta)^{\msf{T}} \msf{P} (\ta) - \bsym{\Si}) e^{(\ta -
  \hat{s}) \msf{A}} \, d \ta + \sum_{k=1}^{2} E_{k} (\infty). 
\end{split}
\end{equation}
Rescaling $\bar{\mbi{d}}$, we define
\begin{equation*}
\mbi{d} = \left( \frac{\Si_{i} \mu_{i}}{\si} \right)_{i=1}^{n-1}, \qquad \Phi (s_{0}) =
\inproav{\mbi{d}}{\int_{s_{0}}^{\tilde{s}} \frac{\msf{P}^{-1} (\ta) \bsym{\phi} (\ta,\mbi{0})}{\subpsup{\psi}{n}{0} (\ta)} \, d \ta},
\end{equation*}
giving
\begin{equation*}
\inproa{\tilde{\mbi{z}}}{\bar{\mbi{d}}} = \frac{\ve \si}{|\la_{1}|} \Phi (s_{0}).
\end{equation*}
The higher-order terms are given by $\mscr{E}_{1} = E_{1} (\infty)$, $\mscr{E}_{2} = E_{2}
(\infty)$,
\begin{equation*}
\mscr{E}_{3} = \inproav{\hat{\mbi{z}}}{\int_{\hat{s}}^{\infty} (\mbi{b}_{1} (\ta)^{\msf{T}}
\msf{P} (\ta) - \bsym{\Si}) e^{(\ta - \hat{s}) \msf{A}} \, d \ta}, \qquad
\mscr{E}_{4} = -\int_{\tilde{s}}^{\hat{s}} b_{0} (\ta) \, d \ta, \qquad \mscr{E}_{5} =
\inproa{\bsym{\xi}_{2} (s_{0}, \mbi{0})}{\bar{\mbi{d}}}.
\end{equation*}
Setting $\mscr{E} = \sum_{k=1}^{5} \mscr{E}_{k}$ and substituting
into~\eqref{e:slform_pre}, we obtain the final form of the singular limit:
\begin{equation}\label{e:slform}
\hat{t} (a) + 1 = \int_{s_{0}}^{g (s_{0}, a)} b_{0} (\ta) \, d \ta + \frac{\ve
  \si}{|\la_{1}|} \Phi (s_{0}) + \mscr{E}.
\end{equation}

\begin{proposition}\label{p:e_bound}
There exists a system constant $K_{1} > 0$ such that the following hold.
\begin{align*}
\norm{\mscr{E}_{1}}_{C^{3} (\mbb{S})} &\leqs K_{1} \frac{\si
  \ve}{|\la_{1}|} \left( \frac{\ve}{|\la_{1}|} \right)\\
\norm{\mscr{E}_{2}}_{C^{3} (\mbb{S})} &\leqs K_{1} \frac{\si
  \ve}{|\la_{1}|} \left( \frac{\ve}{\si} \right)\\
\norm{\mscr{E}_{3}}_{C^{3} (\mbb{S})} &\leqs K_{1} \frac{\si
  \ve}{|\la_{1}|} (|\la_{1}|)\\
\norm{\mscr{E}_{4}}_{C^{3} (\mbb{S})} &\leqs K_{1} \frac{\si
  \ve}{|\la_{1}|} \left( \frac{|\la_{1}|}{\si} \right)\\
\norm{\mscr{E}_{5}}_{C^{3} (\mbb{S})} &\leqs K_{1} \frac{\si
  \ve}{|\la_{1}|} (|\la_{1}|)
\end{align*}
\end{proposition}

\section{Theory of rank one attractors}\label{s:r1t}

Let $D$ denote the closed unit disk in $\RR^{n-1}$ and let $M = \Sbb^{1} \times D$.
We consider a family of maps $G_{\mbf{a}, b} : M \to M$, where $\mbf{a} =
(a_{1}, \ldots, a_{k}) \in \mscr{V}$ is a vector of parameters and $b \in B_{0}$
is a scalar parameter.  Here $\mscr{V} = \mscr{V}_{1} \times \cdots \times \mscr{V}_{k} \subset
\RR^{k}$ is a product of intervals and $B_{0} \subset \RR \setminus \{ 0 \}$ is
a subset of $\RR$ with an accumulation point at $0$.  Points in $M$ are denoted
by $(x,y)$ with $x \in \Sbb^{1}$ and $y \in D$.  Rank one theory postulates the
following.

\begin{list}{\bfseries (\thegrealm)}
{
\usecounter{grealm}
\setlength{\topsep}{1.5ex plus 0.2ex minus 0.2ex}
\setlength{\labelwidth}{0.9cm}
\setlength{\leftmargin}{1.1cm}
\setlength{\labelsep}{0.2cm}
\setlength{\rightmargin}{0.0cm}
\setlength{\parsep}{0.5ex plus 0.2ex minus 0.1ex}
\setlength{\itemsep}{0ex plus 0.2ex}
}
\item
{\bfseries Regularity conditions.}\label{li:g1}
\begin{enumerate}[\bfseries (a)]
\item
For each $b \in B_{0}$, the function $(x,y,\mbf{a}) \mapsto G_{\mbf{a},b} (x,y)$
is $C^{3}$.\label{li:g1a}
\item
Each map $G_{\mbf{a},b}$ is an embedding of $M$ into itself.\label{li:g1b}
\item
There exists $K_{D} > 0$ independent of $\mbf{a}$ and $b$ such that for all
$\mbf{a} \in \mscr{V}$, $b \in B_{0}$, and $z$, $z' \in M$, we have\label{li:g1c}
\begin{equation*}
\frac{|\det DG_{\mbf{a},b} (z)|}{|\det DG_{\mbf{a},b} (z')|} \leqs K_{D}.
\end{equation*}
\end{enumerate}
\item
{\bfseries Existence of a singular limit.}\label{li:g2} For $\mbf{a} \in \mscr{V}$,
there exists a map $G_{\mbf{a},0} : M \to \Sbb^{1} \times \{ 0 \}$ such that the
following holds.  For every $(x,y) \in M$ and $\mbf{a} \in \mscr{V}$, we have
\begin{equation*}
\lim_{b \to 0} G_{\mbf{a}, b} (x,y) = G_{\mbf{a}, 0} (x,y)
\end{equation*}     
Identifying $\Sbb^{1} \times \{ 0 \}$ with
$\Sbb^{1}$, we refer to $G_{\mbf{a},0}$ and the restriction $f_{\mbf{a}} : \Sbb^{1}
\to \Sbb^{1}$ defined by $f_{\mbf{a}} (x) = G_{\mbf{a},0} (x,0)$ as the {\bfseries
\itshape singular limit} of $G_{\mbf{a},b}$.
\item
\label{li:g3}
{\mathversion{bold} $C^{3}$ }{\bfseries convergence to the singular limit.}
We select a special index $j \in \{ 1, \ldots, k \}$.  Fix
$a_{i} \in \mscr{V}_{i}$ for $i \neq j$.  For every such choice of parameters $a_{i}$, the maps $(x, y, a_{j})
\mapsto G_{\mbf{a}, b} (x, y)$ converge in the $C^{3}$ topology to $(x, y,
a_{j}) \mapsto G_{\mbf{a}, 0} (x, y)$ on $M \times \mscr{V}_{j}$ as $b \to 0$.
\item
{\bfseries Existence of a sufficiently expanding map within the singular
limit.}\label{li:g4} 
There exists $\mbf{a}^{*} = (\subsup{a}{1}{*}, \ldots, \subsup{a}{k}{*}) \in
\mscr{V}$ such that $f_{\mbf{a}^{*}} \in \mscr{M}$, where $\mscr{M}$ is the set of
Misiurewicz-type maps defined in Definition~\ref{d:mis} below.
\item
{\bfseries Parameter transversality.}\label{li:g5}
Let $C_{\mbf{a}^{*}}$ denote the critical set of $f_{\mbf{a}^{*}}$.  For $a_{j} \in
\mscr{V}_{j}$, define the vector $\tilde{\mbf{a}}_{j} \in \mscr{V}$ by
$\tilde{\mbf{a}}_{j} = (\subsup{a}{1}{*}, \ldots, \subsup{a}{j-1}{*}, a_{j},
\subsup{a}{j+1}{*}, \ldots, \subsup{a}{k}{*})$.  We say that the family $\{ f_{\mbf{a}}
\}$ satisfies the {\bfseries \itshape parameter transversality} condition with respect to
parameter $a_{j}$ if the following holds.  For each $x \in C_{\mbf{a}^{*}}$, let $p =
f_{\mbf{a}^{*}} (x)$ and let $x(\tilde{\mbf{a}}_{j})$ and $p(\tilde{\mbf{a}}_{j})$ denote
the continuations of $x$ and $p$, respectively, as the parameter $a_{j}$ varies around
$\subsup{a}{j}{*}$.  The point $p(\tilde{\mbf{a}}_{j})$ is the unique point such that
$p(\tilde{\mbf{a}}_{j})$ and $p$ have identical symbolic itineraries under
$f_{\tilde{\mbf{a}}_{j}}$ and $f_{\mbf{a}^{*}}$, respectively.  We have
\begin{equation*}
\left. \frac{d}{d a_{j}} f_{\tilde{\mbf{a}}_{j}} (x(\tilde{\mbf{a}}_{j}))
\right|_{a_{j} = \subsup{a}{j}{*}} \neq \left. \frac{d}{d a_{j}}
p(\tilde{\mbf{a}}_{j}) \right|_{a_{j} = \subsup{a}{j}{*}}.
\end{equation*}
\item
{\bfseries Nondegeneracy at \textquoteleft turns'.}\label{li:g6} 
For each $x \in C_{\mbf{a}^{*}}$, there exists $1 \leqs m \leqs n-1$ such
that
\begin{equation*}
\left. \pdfop{y_{m}} G_{\mbf{a}^{*}, 0} (x,y) \right|_{y = 0} \neq 0.
\end{equation*}
\item
{\bfseries Conditions for mixing.}\label{li:g7}
\begin{enumerate}[\bfseries (a)]
\item
We have $e^{\tfrac{1}{3} \lambda_{0}} > 2$, where $\lambda_{0}$ is defined
within Definition~\ref{d:mis}.\label{li:g7a}
\item
\label{li:g7b}
Let $J_{1}, \ldots, J_{r}$ be the intervals of monotonicity of
$f_{\mbf{a}^{*}}$.  Let $Q = (q_{im})$ be the matrix of `allowed transitions' defined by
\begin{equation*}
q_{im} =
\begin{cases}
1, &\text{if } f_{\mbf{a}^{*}} (J_{i}) \supset J_{m},\\
0, &\text{otherwise}.
\end{cases}
\end{equation*}
There exists $N > 0$ such that $Q^{N} > 0$.
\end{enumerate}
\end{list}
We now define the family $\mscr{M}$.

\begin{definition}\label{d:mis}
  We say that $f \in C^{2} (\Sbb^{1}, \RR)$ is a Misiurewicz map and we write $f \in
  \mscr{M}$ if the following hold for some neighborhood $U$ of the critical set $C = C(f)
  = \{ x \in \Sbb^{1} : f'(x) = 0 \}$.
\begin{enumerate}[(A)]
\item
\textbf{(Outside of} {\mathversion{bold} $U$}\textbf{)}  There
exist $\lambda_{0} > 0$, $M_{0} \in \ZZ^{+}$, and $0 < d_{0} \leqs 1$ such
that\label{li:ma}
\begin{enumerate}[(1)]
\item
for all $m \geqs M_{0}$, if $f^{i} (x) \notin U$ for $0 \leqs i \leqs m-1$, then
$|(f^{m})' (x)| \geqs e^{\lambda_{0} m}$,\label{li:ma1}
\item
for any $m \in \ZZ^{+}$, if $f^{i} (x) \notin U$ for $0 \leqs i \leqs m-1$ and
$f^{m} (x) \in U$, then $|(f^{m})' (x)| \geqs d_{0} e^{\lambda_{0}
m}$.\label{li:ma2}
\end{enumerate}
\item
\textbf{(Critical orbits)}  For all $c \in C$ and $i > 0$, $f^{i} (c) \notin
U$.\label{li:mb}
\item
\textbf{(Inside} {\mathversion{bold} $U$}\textbf{)}\label{li:mc}
\begin{enumerate}[(1)]
\item
We have $f''(x) \neq 0$ for all $x \in U$, and\label{li:mc1}
\item
for all $x \in U \setminus C$, there exists $p_{0} (x) > 0$ such that $f^{i} (x)
\notin U$ for all $i < p_{0} (x)$ and $|(f^{p_{0} (x)})' (x)| \geqs d_{0}^{-1}
e^{\tfrac{1}{3} \lambda_{0} p_{0} (x)}$.\label{li:mc2}
\end{enumerate}
\end{enumerate}
\end{definition}

Rank one theory states that given a family $\{ G_{\mbf{a},b} \}$
satisfying~\textbf{\pref{li:g1}--\pref{li:g6}}, a measure-theoretically significant subset
of this family consists of maps admitting attractors with strong chaotic and stochastic
properties.  We formulate the precise results and we then describe the properties that the
attractors possess.

\begin{theorem}[\cite{WqYls2001, WqYls2008}]\label{t:wyt1}
  Suppose the family $\{ G_{\mbf{a},b} \}$
  satisfies~\bpref{li:g1},~\bpref{li:g2},~\bpref{li:g4}, and~\bpref{li:g6}.  The following
  holds for all $1 \leqs j \leqs k$ such that the parameter $a_{j}$
  satisfies~\bpref{li:g3} and~\bpref{li:g5}.  For all sufficiently small $b \in B_{0}$,
  there exists a subset $\De_{j} \subset \mscr{V}_{j}$ of positive Lebesgue measure such
  that for $a_{j} \in \De_{j}$, $G_{\tilde{\mbf{a}}_{j}, b}$ admits a strange attractor
  $\Om$ with properties~\bpref{li:sa1},~\bpref{li:sa2}, and~\bpref{li:sa3}.
\end{theorem}

\begin{theorem}[\cite{WqYls2001, WqYls2002, WqYls2008}]\label{t:wyt2}
In the sense of Theorem~\ref{t:wyt1},
\begin{equation*}
\text{\textbf{\pref{li:g1}--\pref{li:g7}}} \Longrightarrow
\text{\textbf{\pref{li:sa1}--\pref{li:sa4}}}.
\end{equation*}
\end{theorem}

\begin{remark}
  The proof of Theorem~\ref{t:wyt1} for the special case $n=2$ appears
  in~\cite{WqYls2001}.  The additional component~$\text{\bpref{li:g7}} \Rightarrow
  \text{\bpref{li:sa4}}$ in Theorem~\ref{t:wyt2} is proved in~\cite{WqYls2002}.  For
  general $n$, Wang and Young~\cite{WqYls2008} prove the existence of an SRB measure for
  $G_{\tilde{\mbi{a}}_{j},b}$ if $a_{j} \in \De_{j}$.  The complete proofs
  of~\bpref{li:sa1}--\bpref{li:sa3} (and~\bpref{li:sa4} assuming~\bpref{li:g7}) for
  $G_{\tilde{\mbi{a}}_{j},b}$ with $a_{j} \in \De_{j}$ will appear in~\cite{WqYls2009} for
  general $n$.
\end{remark}

We now describe~\bpref{li:sa1}--\bpref{li:sa4} precisely.  Write $G =
G_{\tilde{\mbi{a}}_{j},b}$.

\begin{list}{\bfseries (\thesarealm)}
{
\usecounter{sarealm}
\setlength{\topsep}{1.5ex plus 0.2ex minus 0.2ex}
\setlength{\labelwidth}{1.15cm}
\setlength{\leftmargin}{1.35cm}
\setlength{\labelsep}{0.2cm}
\setlength{\rightmargin}{0.0cm}
\setlength{\parsep}{0.5ex plus 0.2ex minus 0.1ex}
\setlength{\itemsep}{0ex plus 0.2ex}
}
\item {\bfseries Positive Lyapunov exponent.}\label{li:sa1} Let $U$ denote the basin of
  attraction of the attractor $\Om$.  This means that $U$ is an open set satisfying
  $G(\oline{U}) \subset U$ and
\begin{equation*}
\Om = \bigcap_{m=0}^{\infty} T^{m} (\oline{U}).
\end{equation*}
For almost every
$z \in U$ with respect to Lebesgue measure, the orbit of $z$ has a
positive Lyapunov exponent.  That is,
\begin{equation*}
\lim_{m \to \infty} \frac{1}{m} \log \| D G^{m} (z) \| > 0.
\end{equation*}
\item
{\bfseries Existence of SRB measures and basin property.}\label{li:sa2}
\begin{enumerate}[\bfseries (a)]
\item
The map $G$ admits at least one and at most finitely many ergodic SRB measures
each one of which has no zero Lyapunov exponents.  Let $\nu_{1}, \cdots, \nu_{r}$
denote these measures.\label{li:sa2a}
\item
For Lebesgue-a.e. $z \in U$, there exists $j(z) \in \{ 1, \ldots, r \}$ such
that for every continuous function $\vp : U \to \RR$,\label{li:sa2b}
\begin{equation*}
\frac{1}{m} \sum_{i=0}^{m-1} \vp (G^{i} (x,y)) \to \int \vp \, d \nu_{j(z)}.
\end{equation*}
\end{enumerate}
\item
{\bfseries Statistical properties of dynamical observations.}\label{li:sa3}
\begin{enumerate}[\bfseries (a)]
\item
For every ergodic SRB measure $\nu$ and every H\"{o}lder continuous function
$\vp : \Om \to \RR$, the sequence $\{ \vp \circ G^{i} : i \in \ZZ^{+} \}$ obeys
a central limit theorem.  That is, if $\int \vp \, d \nu = 0$, then the
sequence\label{li:sa3a}
\begin{equation*}
\frac{1}{\sqrt{m}} \sum_{i=0}^{m-1} \vp \circ G^{i}
\end{equation*}
converges in distribution (with respect to $\nu$) to the normal distribution.  The variance of the
limiting normal distribution is strictly positive unless $\vp = \psi
\circ G - \psi$ for some $\psi \in L^{2} (\nu)$.
\item
\label{li:sa3b}
Suppose that for some $N \geqs 1$, $G^{N}$ has an SRB measure $\nu$ that is
mixing.  Then given a H\"{o}lder exponent $\eta$, there exists $\tau = \tau
(\eta) < 1$ such that for all H\"{o}lder $\vp$, $\psi : \Om \to \RR$ with
H\"{o}lder exponent $\eta$, there exists $K = K(\vp, \psi)$ such that for all $m
\in \NN$,
\begin{equation*}
\left| \int (\vp \circ G^{mN}) \psi \, d \nu - \int \vp \, d \nu \int \psi \, d
\nu \right| \leqs K(\vp, \psi) \tau^{m}.
\end{equation*}
\end{enumerate}
\item
{\bfseries Uniqueness of SRB measures and ergodic properties.}\label{li:sa4}
\begin{enumerate}[\bfseries (a)]
\item
The map $G$ admits a unique (and therefore ergodic) SRB measure $\nu$,
and\label{li:sa4a}
\item
the dynamical system $(G, \nu)$ is mixing, or, equivalently, isomorphic to a
Bernoulli shift.\label{li:sa4b}
\end{enumerate}
\end{list}

\section{Verification of the rank one hypotheses}\label{s:sl_analysis}

We view the singular limit $\{ G_{a,0} : a \in \mbb{S} \}$ as a function of $3$
parameters: $\ve$, $\si$, and $\la_{1}$.  We show that the family $\{ G_{a,m^{-1}} : a \in
\mbb{S}, \; \: m \in \mbb{Z}^{+} \}$ satisfies~\bpref{li:g1}--\bpref{li:g7} if the
parameters $\ve$, $\si$, and $\la_{1}$ satisfy certain scaling assumptions.

\subsection{$\mbold{1D}$ analysis: verification of~\bpref{li:g4},~\bpref{li:g5},
  and~\bpref{li:g7}}\label{ss:sl_1d}

Recall that $g(s,a)$ is defined implicitly by
\begin{equation*}
\hat{t} (a) + 1 = \int_{s}^{g(s,a)} b_{0} (\ta) \, d \ta + \frac{\ve \si}{|\la_{1}|} \Phi
(s) + \mscr{E}.
\end{equation*}
Defining $f_{a} (s) = g(s,a)$, $\La = \frac{\ve \si}{|\la_{1}|}$, and $\Psi (s) = \Phi (s) + \La^{-1} \mscr{E}$,
the singular limit becomes
\begin{equation}\label{e:sl_form2}
\hat{t} (a) + 1 = \int_{s}^{f_{a} (s)} b_{0} (\ta) \, d \ta + \La \Psi (s).
\end{equation}
For a map $f : \mbb{S} \to \mbb{S}$ and $\de > 0$, let $C(f) = \{ s : f'(s) = 0
\}$ and let $C_{\de} (f) = \{ s : |s - \hat{s}| < \de \text{ for some } \hat{s} \in C(f)
\}$.  We assume the following about $\Psi$: there exist positive constants $K_{2}$,
$d_{0}$, $d_{1}$, and $d_{2}$, and a constant $\de_{0}$ satisfying $0 < \de_{0} < \frac{1}{2} d_{1}$,
such that the following hold.
\begin{list}{\bfseries (\thearealm)}
{
\usecounter{arealm}
\setlength{\topsep}{1.5ex plus 0.2ex minus 0.2ex}
\setlength{\labelwidth}{0.9cm}
\setlength{\leftmargin}{1.1cm}
\setlength{\labelsep}{0.2cm}
\setlength{\rightmargin}{0.0cm}
\setlength{\parsep}{0.5ex plus 0.2ex minus 0.1ex}
\setlength{\itemsep}{0ex plus 0.2ex}
}
\item
\label{li:a1}
$\norm{\Psi}_{C^{3} (\mbb{S})} < K_{2}$
\item
\label{li:a2}
$|\Psi''(s)| > d_{0}$ for $s \in C_{\de_{0}} (\Psi)$
\item
\label{li:a3}
If $\Psi'(s_{1}) = \Psi'(s_{2}) = 0$ and $s_{1} \neq s_{2}$, then $|s_{1}-s_{2}| > d_{1}$.
\item
\label{li:a4}
$|\Psi'(s)| > d_{2}$ for $s \in \mbb{S} \setminus C_{\de_{0}} (\Psi)$
\end{list}
Because $\Phi$ is a Morse function, Proposition~\ref{p:e_bound} implies that
assumptions~\pref{li:a1}--\pref{li:a4} are satisfied if $\si \geqs 1$, $|\la_{1}|$ is
sufficiently small, and $\frac{\ve}{|\la_{1}|}$ is sufficiently small.

We now compare the map $f_{a}$ to the map $\Psi$.  Let $\{ \bar{v}_{1}, \ldots,
\bar{v}_{q_{0}} \}$ be the set of critical points of $\Psi$.  Set $\xi =
\La^{-\frac{3}{4}}$.

\begin{lemma}\label{l:morse_estm}
  There exists $\La_{0} > 0$ and positive constants $K_{3}$, $K_{4}$, and $K_{5}$ such
  that the following hold for fixed $\La > \La_{0}$.
\begin{enumerate}[(a)]
\item
\label{li:morse_1}
$C (f_{a}) = \{ v_{1}, \ldots, v_{q_{0}} \}$ with $|v_{i} - \bar{v}_{i}| < K_{3} \La^{-1}$
for $1 \leqs i \leqs q_{0}$
\item
\label{li:morse_2}
$|f_{a}''(s)| > K_{4} \La$ for all $s \in C_{\xi} (f_{a})$
\item
\label{li:morse_3}
$|f_{a}'(s)| > K_{5} \La^{\frac{1}{4}}$ for all $s \in \mbb{S} \setminus
C_{\frac{1}{2} \xi} (f_{a})$
\end{enumerate}
\end{lemma}

\begin{proof}[Proof of Lemma~\ref{l:morse_estm}]
Differentiating~\eqref{e:sl_form2} with respect to $s$, we obtain
\begin{equation}\label{e:sla_e0}
b_{0}(s) - \La \Psi'(s) = b_{0}(f_{a}(s)) f_{a}'(s).
\end{equation}
Setting $f_{a}'(s) = 0$ gives $b_{0}(s) = \La \Psi'(s)$.  Since $b_{0}$ is bounded above
and bounded away from $0$,~\pref{li:a2}--\pref{li:a4} imply~\pref{li:morse_1}.  Solving
for $f_{a}'(s)$, we have
\begin{equation}\label{e:sla_e1}
f_{a}'(s) = \frac{b_{0}(s) - \La \Psi'(s)}{b_{0} (f_{a}(s))}.
\end{equation}
On $\mbb{S} \setminus C_{\frac{1}{2} \xi} (f_{a})$ we have $|\Psi'(s)| > K \xi$
using~\pref{li:morse_1}, \pref{li:a2}, and~\pref{li:a4}.  Estimate~\pref{li:morse_3} now
follows from~\eqref{e:sla_e1}.

Differentiating~\eqref{e:sla_e0} with respect to $s$, we obtain
\begin{equation}\label{e:sla_e2}
b_{0}'(s) - \La \Psi''(s) - b_{0}'(f_{a}(s)) [f_{a}'(s)]^{2} = b_{0}(f_{a}(s)) f_{a}''(s).
\end{equation}
For all $s \in C_{\xi} (f_{a})$, we have $|\Psi'(s)| < K \xi$ by~\pref{li:a1}
and~\pref{li:morse_1}.  This implies that $|f_{a}'(s)| < K \La^{\frac{1}{4}}$ on $C_{\xi}$
using~\eqref{e:sla_e1}.  Therefore the second term on the left side of~\eqref{e:sla_e2}
dominates and~\pref{li:morse_2} holds.
\end{proof}

\subsubsection{Critical curves}

Assume $\La > \La_{0}$ and let $\De \subset \mbb{S}$ be a parameter interval.  For $a
\in \De$, we have $C(f_{a}) = \{ v_{1} (a), \ldots, v_{q_{0}} (a) \}$ by
Lemma~\ref{l:morse_estm}.  Write $\psup{\ga}{i} (a) = v_{i} (a)$ for $1 \leqs i \leqs
q_{0}$.  For $1 \leqs k \leqs q_{0}$ and $i \in \mbb{N}$, define $\subpsup{\ga}{i}{k} (a)
\defas f_{a}^{i} (\psup{\ga}{k} (a))$.

Differentiating $\subpsup{\ga}{1}{k} (a) = f_{a} (\psup{\ga}{k} (a)) = f(\psup{\ga}{k}
(a), a)$ with respect to $a$,
we have
\begin{equation}\label{e:dcc_1}
\begin{aligned}
\tdfop{a} \subpsup{\ga}{1}{k} (a) &= \pdf{f}{s} (\psup{\ga}{k} (a),a) \cdot \tdfop{a}
\psup{\ga}{k} (a) + \pdf{f}{a} (\psup{\ga}{k} (a),a)\\
&= \pdf{f}{a} (\psup{\ga}{k} (a),a).
\end{aligned}
\end{equation}
Differentiating~\eqref{e:sl_form2} with respect to $a$ and using the fact that $\tdfop{a}
\hat{t} (a) = b_{0}(a)$, we obtain
\begin{equation}\label{e:partial_a_f}
\pdfop{a} f(s,a) = \frac{b_{0}(a)}{b_{0}(f(s,a))}.
\end{equation}
Thus 
\begin{equation*}
\tdfop{a} \subpsup{\ga}{1}{k} (a) \geqs \frac{\min_{s \in \mbb{S}}
  b_{0} (s)}{\max_{s \in \mbb{S}} b_{0}(s)} > 0.
\end{equation*}
More generally, an estimate on $\tdfop{a} \subpsup{\ga}{i+1}{k} (a)$ for $i \in \mbb{N}$
follows from the recursive formula
\begin{equation}\label{e:sla_e4}
\tdfop{a} \subpsup{\ga}{i+1}{k} (a) = \pdf{f}{s} (\subpsup{\ga}{i}{k} (a),a) \cdot
\tdfop{a} \subpsup{\ga}{i}{k} (a) + \pdf{f}{a} (\subpsup{\ga}{i}{k} (a),a).
\end{equation}

\begin{lemma}[Growth estimate for derivatives of critical curves]\label{l:d_cc}
There exists $\La_{1} \geqs \La_{0}$ such that the following holds for all $\La >
\La_{1}$.  For any $k \in \{ 1, \ldots, q_{0} \}$ and $i \in \mbb{N}$ such that
$\subpsup{\ga}{j}{k} (a) \in \mbb{S} \setminus C_{\xi} (\Psi)$ for all $1 \leqs j
\leqs i$, then
\begin{equation}\label{e:d_cc}
\left| \tdfop{a} \subpsup{\ga}{i+1}{k} (a) \right| > \left( \frac{K_{5}}{2}
  \La^{\frac{1}{4}} \right)^{i} > \La^{\frac{i}{5}}.
\end{equation}
\end{lemma}

\begin{proof}[Proof of Lemma~\ref{l:d_cc}]
Estimate~\eqref{e:d_cc} follows from~\eqref{e:sla_e4}, estimate~\pref{li:morse_3} from
Lemma~\ref{l:morse_estm}, and the fact that for all $s \in \mbb{S}$ and $a \in \De$ we
have
\begin{equation*}
\pdfop{a} f(s,a) \leqs \frac{\max_{s \in \mbb{S}} b_{0} (s)}{\min_{s \in \mbb{S}}
  b_{0} (s)} \defasr K_{6}.
\end{equation*}
\end{proof}

\begin{lemma}[Distortion estimate for critical curves]\label{l:cc_distor}
  There exists $\La_{2} \geqs \La_{1}$ and $D_{1} > 0$ such that the following holds for
  all $\La > \La_{2}$.  For any $k \in \{ 1, \ldots, q_{0} \}$ and any $n \geqs 2$, let $\De$
  be a parameter interval such that
\begin{enumerate}[(a)]
\item
\label{li:distor_h1}
$\subpsup{\ga}{i}{k} (\De) \subset \mbb{S} \setminus C_{\xi} (\Psi)$ for $1 \leqs i \leqs
n-1$, and
\item
\label{li:distor_h2}
$\ell (\subpsup{\ga}{n-1}{k} (\De)) < \xi$ ($\ell$ denotes Lebesgue measure on
$\mbb{S}$).
\end{enumerate}
Then for all $a, \hat{a} \in \De$, we have
\begin{equation}\label{e:cc_distortion}
\left| \frac{\tdfop{a} \subpsup{\ga}{n}{k} (a)}{\tdfop{a} \subpsup{\ga}{n}{k} (\hat{a})}
\right| < D_{1}.
\end{equation}
If $n=1$, then~\eqref{e:cc_distortion} holds for all $k \in \{ 1, \ldots, q_{0} \}$ and
for all $a, \hat{a} \in \mbb{S}$.
\end{lemma}

\begin{proof}[Proof of Lemma~\ref{l:cc_distor}]
For $n=1$ and $a, \hat{a} \in \mbb{S}$, the estimate
\begin{equation*}
\left| \frac{\tdfop{a} \subpsup{\ga}{1}{k} (a)}{\tdfop{a} \subpsup{\ga}{1}{k} (\hat{a})}
\right| < K_{6}^{2}
\end{equation*}
follows from~\eqref{e:dcc_1} and~\eqref{e:partial_a_f}.  For $n \geqs 2$ and $a, \hat{a}
\in \De$, let $s_{i} = \subpsup{\ga}{i}{k} (a)$ and $\hat{s}_{i} = \subpsup{\ga}{i}{k}
(\hat{a})$.  We have
\begin{equation*}
\left| \frac{\tdfop{a} s_{i}}{\tdfop{a} \hat{s}_{i}} \right| = \left|
  \frac{f_{a}'(s_{i-1}) \cdot \tdfop{a} s_{i-1} + \pdfop{a} f_{a}
    (s_{i-1})}{f_{\hat{a}}'(\hat{s}_{i-1}) \cdot \tdfop{a} \hat{s}_{i-1} + \pdfop{a}
    f_{\hat{a}} (\hat{s}_{i-1})} \right| = \left| \frac{f_{a}'(s_{i-1}) \cdot \tdfop{a}
    s_{i-1}}{f_{\hat{a}}'(\hat{s}_{i-1}) \cdot \tdfop{a} \hat{s}_{i-1}} \right| \left( 1 +
\mcal{O} (\La^{-\frac{1}{5} (i-1)}) \right).
\end{equation*}
This implies the estimate
\begin{align*}
\log \left| \frac{\tdfop{a} s_{n}}{\tdfop{a} \hat{s}_{n}} \right| &= \log \left|
  \frac{\tdfop{a} s_{1}}{\tdfop{a} \hat{s}_{1}} \right| + \sum_{i=1}^{n-1} \log \left|
  \frac{f_{a}'(s_{i})}{f_{\hat{a}}' (\hat{s}_{i})} \right| + \sum_{i=1}^{n-1} \log \left(
  1 + \mcal{O} (\La^{-\tfrac{i}{5}}) \right)\\
&\leqs \sum_{i=1}^{n-1} \frac{|f_{a}'(s_{i}) -
  f_{\hat{a}}'(\hat{s}_{i})|}{|f_{\hat{a}}'(\hat{s}_{i})|} + \mcal{O} (1).
\end{align*}
The equality
\begin{equation*}
|f_{a}'(s_{i}) - f_{\hat{a}}'(\hat{s}_{i})| = \left| \frac{1}{b_{0}(f_{a} (s_{i}))} \left(
    (b_{0} (s_{i+1}) - b_{0} (\hat{s}_{i+1})) f_{\hat{a}}'(\hat{s}_{i}) + \La
    (\Psi'(s_{i}) - \Psi'(\hat{s}_{i})) + (b_{0} (\hat{s}_{i}) - b_{0} (s_{i})) \right)
\right|
\end{equation*}
implies the estimate
\begin{align*}
\log \left| \frac{\tdfop{a} s_{n}}{\tdfop{a} \hat{s}_{n}} \right| &\leqs K
\sum_{i=1}^{n-1} |s_{i+1} - \hat{s}_{i+1}| + K \La^{\frac{3}{4}} \sum_{i=1}^{n-1} |s_{i} -
  \hat{s}_{i}| + K \La^{-\frac{1}{4}} \sum_{i=1}^{n-1} |\hat{s}_{i} - s_{i}| + \mcal{O}
  (1)\\
&= K \left( |s_{n} - \hat{s}_{n}| + \sum_{i=2}^{n-1} |s_{i} - \hat{s}_{i}| \right) + K \La^{\frac{3}{4}} \sum_{i=1}^{n-1} |s_{i} -
  \hat{s}_{i}| + K \La^{-\frac{1}{4}} \sum_{i=1}^{n-1} |\hat{s}_{i} - s_{i}| + \mcal{O}
  (1)\\
&\leqs K |s_{n} - \hat{s}_{n}| + K |s_{n-1} - \hat{s}_{n-1}| \sum_{i=0}^{n-3}
\La^{-\frac{1}{5} i} + K |s_{n-1} - \hat{s}_{n-1}| (\La^{\frac{3}{4}} +
\La^{-\frac{1}{4}}) \sum_{i=0}^{n-2} \La^{-\frac{1}{5} i} + \mcal{O} (1)\\
&= \mcal{O} (1).
\end{align*}
\end{proof}

\subsubsection{Verification of~\bpref{li:g4}: Definition~\ref{d:mis}\pref{li:mb}}

We prove the existence of a parameter $a^{*}$ such that $f_{a^{*}}$ satisfies
Definition~\ref{d:mis}\pref{li:mb}.  We will then show that if $\La$ is sufficiently
large, then for any parameter $a$, if $f_{a}$ satisfies
Definition~\ref{d:mis}\pref{li:mb}, then $f_{a} \in \mscr{M}$.  

\begin{proposition}\label{p:crit_orbits}
There exists $\La_{3} \geqs \La_{2}$ such that if $\La \geqs \La_{3}$ and $\De \subset
\mbb{S}$ is a parameter interval satisfying $\ell (\De) = 3 D_{1} K_{6} q_{0} \xi$,
then there exists $a^{*} \in \De$ such that for all $c \in C(f_{a^{*}})$, $f_{a^{*}}^{n}
(c) \in \mbb{S} \setminus C_{\xi} (\Psi)$ for all $n \in \mbb{N}$.
\end{proposition}

\begin{proof}[Proof of Proposition~\ref{p:crit_orbits}]
  We inductively construct a nested sequence of parameter intervals $\De = \De_{0} \supset
  \De_{1} \supset \De_{2} \supset \cdots$ such that $a^{*} \in \bigcap_{i=0}^{\infty}
  \De_{i}$ has the desired property.

\begin{definition}\label{d:adm_cfg}
  The $(q_{0}+1)$-tuple $(\De_{n}; i_{1,n}, \ldots, i_{q_{0},n})$ is called an {\bfseries
    \itshape admissible configuration} if $\De_{n}$ is a subinterval of $\De_{0}$ and if
  for every $k \in \{ 1, \ldots, q_{0} \}$, $i_{k,n} \leqs n$ and the following conditions
  are satisfied.
\begin{list}{\bfseries (\themrealm)}
{
\usecounter{mrealm}
\setlength{\topsep}{1.5ex plus 0.2ex minus 0.2ex}
\setlength{\labelwidth}{0.9cm}
\setlength{\leftmargin}{1.1cm}
\setlength{\labelsep}{0.2cm}
\setlength{\rightmargin}{0.0cm}
\setlength{\parsep}{0.5ex plus 0.2ex minus 0.1ex}
\setlength{\itemsep}{0ex plus 0.2ex}
}
\item
\label{li:m1}
$\subpsup{\ga}{i}{k} (\De_{n}) \cap C_{\xi} (\Psi) = \emptyset$ for all $i \leqs i_{k,n}$
\item
\label{li:m2}
For all $a, \hat{a} \in \De_{n}$, we have the distortion estimate
\begin{equation*}
\left| \frac{\tdfop{a} \subpsup{\ga}{i_{k,n}}{k} (a)}{\tdfop{a} \subpsup{\ga}{i_{k,n}}{k}
    (\hat{a})} \right| < D_{1}.
\end{equation*}
\item
\label{li:m3}
$\ell \big( \subpsup{\ga}{i_{k,n} + 1}{k} (\De_{n}) \big) \geqs 3 D_{1} q_{0} \xi$
\end{list}
\end{definition}

We inductively construct admissible configurations for all $n \in \mbb{N}$ such that
$i_{k,n} \to \infty$ as $n \to \infty$ for every $k$.  We begin with $n=1$.  Let
\begin{equation*}
\tilde{d} \defas \min_{\substack{s, t \in C(\Psi)\\ s \neq t}} |s-t|.
\end{equation*}
We assume that $3 D_{1} K_{6}^{2} q_{0} \xi < \frac{1}{2} \tilde{d}$.  Let $i_{k,1} = 1$
for all $k$.  We choose $\De_{1}$ as follows.  We have
\begin{equation*}
\tdfop{a} \subpsup{\ga}{1}{k} (a) = \frac{b_{0} (a)}{b_{0} (\subpsup{\ga}{1}{k} (a))},
\end{equation*}
so
\begin{equation*}
3 D_{1} q_{0} \xi \leqs \ell (\subpsup{\ga}{1}{k} (\De_{0})) \leqs 3
D_{1} K_{6}^{2} q_{0} \xi < \frac{1}{2} \tilde{d}.
\end{equation*}
Consequently, $\subpsup{\ga}{1}{k} (\De_{0})$ meets at most one component of $C_{\xi}
(\Psi)$ and we have 
\begin{equation*}
\frac{\ell ((\subpsup{\ga}{1}{k} | \De_{0})^{-1} (C_{\xi} (\Psi)))}{\ell (\De_{0})} \leqs
\frac{2}{3 q_{0}}.
\end{equation*}
Even in the worst-case scenario in which the $q_{0}$ intervals $\{
(\subpsup{\ga}{1}{k})^{-1} (C_{\xi} (\Psi)) : 1 \leqs k \leqs q_{0} \}$ are evenly spaced
in $\De_{0}$, there exists a subinterval $\De_{1}$ of $\De_{0}$ with $\ell (\De_{1}) \geqs
\frac{D_{1} K_{6} q_{0} \xi}{q_{0} + 1}$ such that $\subpsup{\ga}{1}{k} (\De_{1}) \cap
C_{\xi} (\Psi) = \emptyset$ for all $k$.  Property~\pref{li:m1} holds by design
and~\pref{li:m2} follows from Lemma~\ref{l:cc_distor}.  Property~\pref{li:m3} holds if
$\La$ is such that
\begin{equation*}
\ell (\subpsup{\ga}{2}{k} (\De_{1})) \geqs \La^{\frac{1}{5}} \ell (\De_{1}) \geqs 3 D_{1}
q_{0} \xi.
\end{equation*}

Now assume that for $n \in \mbb{N}$ we are given an admissible configuration $(\De_{n};
i_{1,n}, \ldots, i_{q_{0},n})$.  We construct an admissible configuration at step $n+1$ as
follows.  Partition the set $\{ 1, \ldots, q_{0} \}$ into $2$ sets: $A$, the set of
indices that are `ready to advance', and $\{ 1, \ldots, q_{0} \} \setminus A$, the set of
indices that are not ready to advance.  The index $k$ is in $A$ if~\pref{li:i1}
and~\pref{li:i2} hold:
\begin{list}{\bfseries (\theirealm)}
{
\usecounter{irealm}
\setlength{\topsep}{1.5ex plus 0.2ex minus 0.2ex}
\setlength{\labelwidth}{0.9cm}
\setlength{\leftmargin}{1.1cm}
\setlength{\labelsep}{0.2cm}
\setlength{\rightmargin}{0.0cm}
\setlength{\parsep}{0.5ex plus 0.2ex minus 0.1ex}
\setlength{\itemsep}{0ex plus 0.2ex}
}
\item
\label{li:i1}
$\ell (\subpsup{\ga}{i_{k,n}}{k} (\De_{n})) < \xi$ (distortion estimate holds for the next
iterate)
\item
\label{li:i2}
$\ell (\subpsup{\ga}{i_{k,n} + 1}{k} (\De_{n})) < \frac{1}{2} \tilde{d}$ (image of the
next iterate meets at most one component of $C_{\xi} (\Psi)$)
\end{list}
Suppose that $A \neq \emptyset$.  In this case, set
\begin{equation*}
i_{k,n+1} =
\begin{cases}
i_{k,n} + 1, &\text{if } k \in A;\\
i_{k,n}, &\text{if } k \in \{ 1, \ldots, q_{0} \} \setminus A.
\end{cases}
\end{equation*}
We now find $\De_{n+1}$ so that $(\De_{n+1}; i_{1,n+1}, \ldots, i_{q_{0},n+1})$ is an
admissible configuration.  Let $k \in A$.  Using~\pref{li:m3} and~\pref{li:i2}, we have
\begin{equation*}
3 D_{1} q_{0} \xi \leqs \ell (\subpsup{\ga}{i_{k,n} + 1}{k} (\De_{n})) < \frac{1}{2}
\tilde{d}.
\end{equation*}
This implies that the fraction of $\subpsup{\ga}{i_{k,n} + 1}{k} (\De_{n})$ in $C_{\xi}
(\Psi)$ is bounded above by $\frac{2 \xi}{3 D_{1} q_{0} \xi} = \frac{2}{3 D_{1} q_{0}}$.
Using~\pref{li:i1} and Lemma~\ref{l:cc_distor}, we have
\begin{equation*}
\frac{\ell ((\subpsup{\ga}{i_{k,n} + 1}{k} | \De_{n})^{-1} (C_{\xi} (\Psi)))}{\ell
  (\De_{n})} \leqs \frac{2}{3 q_{0}}.
\end{equation*}
Arguing as in the $n=1$ case, there exists a subinterval $\De_{n+1}$ of $\De_{n}$ such
that $\ell (\De_{n+1}) \geqs \frac{1}{3 (q_{0}+1)} \ell (\De_{n})$ and for all $k \in A$,
$\subpsup{\ga}{i_{k,n} + 1}{k} (\De_{n+1}) \cap C_{\xi} (\Psi) = \emptyset$.  For $k \in
A$,~\pref{li:m1} holds by design and~\pref{li:m2} follows from~\pref{li:i1}.  The
inequality
\begin{equation*}
\ell (\subpsup{\ga}{i_{k,n} + 1}{k} (\De_{n+1})) \geqs \frac{D_{1}^{-1}}{3 (q_{0} + 1)} (3
D_{1} q_{0} \xi) = \frac{q_{0}}{q_{0}+1} \xi
\end{equation*}
implies that~\pref{li:m3} holds if $\La$ is such that
\begin{equation*}
\ell (\subpsup{\ga}{i_{k,n} + 2}{k} (\De_{n+1})) \geqs \left( \frac{q_{0} \xi}{q_{0}+1}
\right) \left( \frac{K_{5}}{2} \La^{\frac{1}{4}} \right) \geqs 3 D_{1} q_{0} \xi.
\end{equation*}

Now let $k \in \{ 1, \ldots, q_{0} \} \setminus A$.  Properties~\pref{li:m1}
and~\pref{li:m2} are inherited from step $n$.  If~\pref{li:i1} fails for index $k$,
then~\pref{li:m2} gives
\begin{equation*}
\ell (\subpsup{\ga}{i_{k,n}}{k} (\De_{n+1})) \geqs \frac{\xi}{3 D_{1} (q_{0} + 1)},
\end{equation*}
so index $k$ satisfies~\pref{li:m3} if $\La$ is such that
\begin{equation*}
\ell (\subpsup{\ga}{i_{k,n} + 1}{k} (\De_{n+1})) \geqs \left( \frac{\xi}{3 D_{1} (q_{0} +
    1)} \right) \left( \frac{K_{5}}{2} \La^{\frac{1}{4}} \right) \geqs 3 D_{1} q_{0} \xi.
\end{equation*}
If~\pref{li:i1} holds but~\pref{li:i2} fails, then~\pref{li:m3} holds for index $k$ if
$\La$ is such that
\begin{equation*}
\ell (\subpsup{\ga}{i_{k,n} + 1}{k} (\De_{n+1})) \geqs \left( \frac{1}{3 D_{1} (q_{0} +
    1)} \right) \left( \frac{1}{2} \tilde{d} \right) \geqs 3 D_{1} q_{0} \xi.
\end{equation*}

If $A = \emptyset$, then let $\De_{n}'$ be the left half of $\De_{n}$.  We claim that
$(\De_{n}'; i_{1,n}, \ldots, i_{q_{0},n})$ is an admissible configuration.  For each index
$k$, properties~\pref{li:m1} and~\pref{li:m2} trivially hold.  Property~\pref{li:m3} is
established (for $\La$ sufficiently large) by arguing as above in the $2$ cases
\begin{enumerate}
\item
\pref{li:i1} does not hold, and
\item
\pref{li:i1} holds but~\pref{li:i2} fails.
\end{enumerate}
Repeat the halving process until $A \neq \emptyset$. 
\end{proof}

\subsubsection{Verification of~\bpref{li:g4}: $f_{a^{*}}$ satisfies
  Definition~\ref{d:mis}\pref{li:mb} $\Rightarrow$ $f_{a^{*}} \in
  \mscr{M}$}\label{ss:bind}

We show that if $\La$ is sufficiently large and $a^{*} \in \mbb{S}$ is as in
Proposition~\ref{p:crit_orbits}, then $f_{a^{*}} \in \mscr{M}$.  This implication is a
consequence of Lemma~\ref{l:morse_estm} and the following binding estimate.

\begin{proposition}\label{p:binding_estm}
There exists $K_{7} > 0$ such that for $\La$ sufficiently large and $a^{*}$ as in
Proposition~\ref{p:crit_orbits}, we have the following.  For $c \in C(f_{a^{*}})$ and $s
\in \mbb{S}$ satisfying $|s-c| \leqs \La^{-\frac{11}{12}}$, let $m(s)$ be the smallest
value of $m \in \mbb{Z}^{+}$ such that $|f_{a^{*}}^{m} (s) - f_{a^{*}}^{m} (c)| >
\frac{1}{2} \xi$.  Then $m(s) > 1$ and
\begin{equation*}
|(f_{a^{*}}^{m(s)})'(s)| \geqs (K_{7} \La)^{\frac{m(s)}{16}}.
\end{equation*}
\end{proposition}

\begin{proof}[Proof of Proposition~\ref{p:binding_estm}]
We begin with a spatial distortion lemma.

\begin{lemma}[Spatial distortion estimate]\label{l:space_dstor}
  There exists $D_{2} \geqs 1$ such that the following holds for all $a \in \mbb{S}$.  For $s, \hat{s} \in
  \mbb{S}$, let $m \in \mbb{Z}^{+}$ be such that $\pi_{i}$, the segment between
  $f_{a}^{i}(s)$ and $f_{a}^{i}(\hat{s})$, satisfies $\ell (\pi_{i}) < \frac{1}{2}
  \xi$ and $\pi_{i} \cap C_{\frac{1}{2} \xi} (f_{a}) = \emptyset$ for all $0 \leqs i <
  m$.  Then
\begin{equation*}
\left| \frac{(f_{a}^{m})'(s)}{(f_{a}^{m})'(\hat{s})} \right| \leqs D_{2}.
\end{equation*}
\end{lemma}

\begin{proof}[Proof of Lemma~\ref{l:space_dstor}]
Writing $s_{i} = f_{a}^{i} (s)$ and $\hat{s}_{i} = f_{a}^{i} (\hat{s})$ and using
Lemma~\ref{l:morse_estm} and its proof, we have
\begin{align*}
\log \left| \frac{(f_{a}^{m})'(s)}{(f_{a}^{m})'(\hat{s})} \right| &=
\sum_{i=0}^{m-1} \log \left| \frac{f_{a}'(s_{i})}{f_{a}'(\hat{s}_{i})} \right|\\
&\leqs \sum_{i=0}^{m-1} \frac{|f_{a}'(s_{i}) -
  f_{a}'(\hat{s}_{i})|}{|f_{a}'(\hat{s}_{i})|}\\
&\leqs K_{5}^{-1} \La^{-\frac{1}{4}} \left( K_{6} + \frac{\La \norm{\Psi'}_{C^{0}
      (\mbb{S})}}{\min_{w \in \mbb{S}} b_{0} (w)} \right) \sum_{i=0}^{m-1} |s_{i}
- \hat{s}_{i}|\\
&\leqs (K \La^{-\frac{1}{4}} + K \La^{\frac{3}{4}}) |s_{m-1} - \hat{s}_{m-1}|
\sum_{i=0}^{m-1} (K_{5} \La^{\frac{1}{4}})^{-i}\\
&= \mcal{O} (1).
\end{align*}
\end{proof}

Returning to the proof of Proposition~\ref{p:binding_estm}, write $f = f_{a^{*}}$.  We
first show that $m(s) > 1$.  We have
\begin{equation*}
|f(s) - f(c)| = \frac{1}{2} |f''(\ze)| (s-c)^{2}
\end{equation*}
for some $\ze$ satisfying $|\ze - c| \leqs \La^{-\frac{11}{12}}$.  Arguing as in the proof
of Lemma~\ref{l:morse_estm}, $|f''(\ze)| \leqs K \La$.  Therefore
\begin{equation*}
|f(s)-f(c)| \leqs K \La^{-\frac{5}{6}} \leqs \frac{\xi}{2}
\end{equation*}
for $\La$ sufficiently large.

Now assume $m(s) > 1$.  Using Lemma~\ref{l:space_dstor}, we have
\begin{align*}
\frac{\xi}{2} &< |f^{m(s)}(s) - f^{m(s)}(c)|\\
&= |(f^{m(s)-1})'(\ze_{1})| \cdot |f(s)-f(c)| \qquad (\text{for some } \ze_{1} \text{
  between } f(s) \text{ and } f(c))\\
&\leqs D_{2} |(f^{m(s)-1})'(f(c))| \cdot |f(s)-f(c)|
\end{align*}
and therefore
\begin{equation}\label{e:bind_1}
\xi < D_{2} |(f^{m(s)-1})'(f(c))| \cdot |f''(\ze)| \cdot (s-c)^{2}.
\end{equation}
Reversing inequality~\eqref{e:bind_1} at time $m(s)-1$, we have
\begin{equation}\label{e:bind_2}
\xi \geqs D_{2}^{-1} |(f^{m(s)-2})'(f(c))| \cdot |f''(\ze)| \cdot (s-c)^{2}.
\end{equation}
Estimating $|(f^{m(s)-1})'(f(c))|$ from below using~\eqref{e:bind_1} gives
\begin{align}
|(f^{m(s)})'(s)| &= |f'(s)-f'(c)| \cdot |(f^{m(s)-1})'(f(s))|
\notag \\
&\geqs D_{2}^{-1} |(f^{m(s)-1})'(f(c))| \cdot |f''(\ze_{4})| \cdot |s-c| \qquad (\text{for
  some } \ze_{4} \text{ between } s \text{ and } c)
\notag \\
&\geqs \frac{\xi}{D_{2}^{2} |s-c|} \left( \frac{|f''(\ze_{4})|}{|f''(\ze)|} \right).
\label{e:bind_3}
\end{align}
Arguing as in the proof of Lemma~\ref{l:morse_estm}, $\frac{|f''(\ze_{4})|}{|f''(\ze)|}
\geqs K > 0$ since $\ze_{4}$ and $\ze$ are between $s$ and $c$.  Using this fact and
estimating $|s-c|^{-1}$ from below using~\eqref{e:bind_2},~\eqref{e:bind_3} implies
\begin{align*}
|(f^{m(s)})'(s)| &\geqs \frac{K \xi}{D_{2}^{2}} \left( \frac{D_{2}^{-1}
    |(f^{m(s)-2})'(f(c))| \cdot |f''(\ze)|}{\xi} \right)^{\frac{1}{2}}\\
&\geqs (K \La)^{\frac{m(s)}{8} - \frac{1}{8}}\\
&\geqs (K \La)^{\frac{m(s)}{16}}.
\end{align*}
\end{proof}

\subsubsection{Verification of~\bpref{li:g5} and~\bpref{li:g7}}

The following lemma facilitates the verification of~\bpref{li:g5}.

\begin{lemma}[\cite{TpTcYls1992, TpTcYls1994}]\label{l:ptreform}
Let $f = f_{a^{*}}$.  Suppose that for all $x \in C(f_{a^{*}})$, we
have
\begin{equation*}
\sum_{k = 0}^{\infty} \frac{1}{|(f^{k})' (f(x))|} < \infty.
\end{equation*}
Then for each $x \in C(f_{a^{*}})$,
\begin{equation*}
\sum_{k = 0}^{\infty} \frac{[(\pdop{a} f_{a}) (f^{k}
(x))]_{a = a^{*}}}{(f^{k})' (f(x))} = \left[ \frac{d}{da}
f_{a} (x(a)) - \frac{d}{da}
p(a) \right]_{a = a^{*}}.
\end{equation*}
\end{lemma}

Property~\bpref{li:g5} follows from Lemma~\ref{l:ptreform} for $\La$ sufficiently large.
To see this, suppose $f_{a^{*}} \in \mscr{M}$ and let $c \in C(f_{a^{*}})$.  For $k \in
\mbb{Z}^{+}$, we have 
\begin{equation*}
|(f_{a^{*}}^{k})'(f(c))| \geqs \left( K_{5} \La^{\frac{1}{4}} \right)^{k}
\end{equation*}
by Lemma~\ref{l:morse_estm}\pref{li:morse_3}.  Since $K_{6}^{-1} \leqs \pdfop{a} f(s,a)
\leqs K_{6}$, we conclude that if $\La$ is sufficiently large, then
\begin{equation*}
\sum_{k=0}^{\infty} \frac{[\pdop{a} f_{a} (f_{a^{*}}^{k} (c))]_{a =
    a^{*}}}{(f_{a^{*}}^{k})'(f_{a^{*}} (c))} \geqs K_{6}^{-1} - \sum_{k=1}^{\infty}
\frac{K_{6}}{\left( K_{5} \La^{\frac{1}{4}} \right)^{k}} > 0.
\end{equation*}

Property~\bpref{li:g7} follows from Lemma~\ref{l:morse_estm} and
Proposition~\ref{p:binding_estm} provided $\La$ is sufficiently large.

\newcommand{\bz}{{\mbi{z}}}
\newcommand{\bh}{{\mbi{h}}}
\newcommand{\bze}{{\mbi{\zeta}}}
\newcommand{\bzero}{{\mbi{0}}}
\newcommand{\bzone}{{\mbi{z}^{(1)}}}
\newcommand{\bztwo}{{\mbi{z}^{(2)}}}
\newcommand{\bdelta}{{\mbi{\delta}}}

\appendix
\section{Some proofs}\label{s:tech_proof}


We assume throughout Section~\ref{s:tech_proof} that $L=1$.  Notice that if $\mbi{V}$
denotes a vector field, then
\begin{equation}\label{e:radial_ode}
\frac{d \mbi z}{ds}=\mbi{V}
\quad\Longrightarrow \quad 
\frac{d \norm{\mbi z}}{ds}=\frac{1}{2\norm{\mbi z}}\frac{d\norm{\mbi z}^2}{ds}=\frac{\mbi{z}}{\norm{\mbi z}}\cdot\frac{d\mbi{z}}{ds} = \frac{\mbi{z}}{\norm{\mbi z}}\cdot \mbi{V}.
\end{equation}
We will use this fact together with the following Gr\"onwall-type inequality:
\begin{lemma}\label{l:gronwall}
Assume that $\beta$ is a constant, the function $\varphi$ is continuous on the interval $[\hat s,\check s]$, and that the function $u$ is differentiable and satisfies $\tdf{u}{s}\leqs \beta u+\varphi$ on $(\hat s, \check s)$. Then, for all $s\in(\hat s, \check s)$,
\beqn
u(s)\leqs u(\hat s)e^{\beta(s-\hat s)}+\int_{\hat s}^s e^{\beta(s-\tau)}\varphi(\tau)\,d\tau.
\eeqn
\end{lemma}
\begin{proof}
Suppose $v(s)=u(\hat s)e^{\beta(s-\hat s)}+\int_{\hat s}^s e^{\beta(s-\tau)}\varphi(\tau)\,d\tau$. Then $v$ satisfies the equation $\tdf{v}{s}(s)=\beta v(s)+\varphi(s)$ with $v(\hat s)=u(\hat s)$. Since $u-v$ is differentiable, $\tdf{}{s}(u-v)\leqs \beta (u-v)$, and $(u-v)(\hat s)=0$, a standard Gr\"onwall argument shows that $u\leqs v$.
\end{proof}
We get immediately
\begin{corollary}\label{c:gronwall}
Suppose that in Lemma~\ref{l:gronwall} $\tdf{u}{s}(s)\leqs \frac{\lambda_1}{2} u+C_0e^{\lambda_1(s-\hat s)}$. Then
\begin{equation*}
u(s)\leqs \left(u(\hat s)+\frac{2C_0}{|\lambda_1|}\right)e^{\frac{\lambda_1}{2}(s-\hat s)}.
\end{equation*}
\end{corollary}


Our first application of a Gr\"onwall inequality is
\begin{lemma}\label{le:z_kick_der}
  Assume $\bz$ solves the forced equation \eqref{e:nf_full_z} with $\bz(s_0)=\bz_0$ and
  fix a constant $K>0$. If $\ve/|\lambda_1|$ is sufficiently small,
\begin{equation}\label{e:z_kick_der}
\norm{\partial^m_{s}\partial^l_{s_0}\bz(s)}\leqs C \ve \qquad (0\leqs l+m\leqs 3)
\end{equation}
as long as $s-s_0\leqs K$. Moreover, 
\begin{equation}\label{e:z_kick_der_z0}
\norm{\partial_{\bz_0} \bz-\msf{1}} \leqs C|\lambda_1|.
\end{equation}
\end{lemma}
\begin{proof}
Equation~\eqref{e:nf_full_z} reads
\begin{equation}\label{e:nf_full_z_v2}
  \frac{d\bz}{ds}=\msf{A}\bz+\bh_3(s,\bz) \qquad\text{with}\qquad
  \bh_3(s,\bz)=\mcal{O}_{s} (\ve) 
  +\mcal{O}_{s,\mbi{z}} (\ve \mbi{z}) + \mcal{O}_{s,\mbi{z}} (\ve^{2}) + \mcal{O}_{s,\mbi{z}} (\norm{\mbi{z}}^{2}).
\end{equation}
Assuming $\norm{\bz}/|\lambda_1|$ and $\ve/|\lambda_1|$ are sufficiently small, \eqref{e:radial_ode} implies
\begin{equation*}
\frac{d\norm{\bz}}{ds}\leqs \frac{\lambda_1}{2}\norm{\bz}+C_0\ve.
\end{equation*}
By Lemma~\ref{l:gronwall}, 
\begin{equation*}
\norm{\bz(s)}\leqs \norm{\bz_0}e^{\frac{\lambda_1}{2}(s-s_0)}+ \frac{2C_0\ve}{|\lambda_1|}\left(1-e^{\frac{\lambda_1}{2}(s-s_0)}\right) \leqs \norm{\bz_0}e^{\frac{\lambda_1}{2}(s-s_0)}+ C_0\ve(s-s_0).
\end{equation*}
For $s-s_0\leqs K$ we get $\norm{\bz(s)}/|\lambda_1|\leqs \norm{\bz_0}/|\lambda_1|+C_0 K\ve/|\lambda_1|$, which proves the assumption legitimate. 

Differentiating \eqref{e:nf_full_z_v2} with respect to $s$ up to two times yields an expression for $\partial^m_{s}\bz(s)$. One immediately obtains
\begin{equation*}
\norm{\partial^m_{s}\bz(s)}\leqs C \ve
\end{equation*}
for $0\leqs m\leqs 3$ and $s-s_0\leqs K$. 

Equation~\eqref{e:nf_full_z_v2} implies
\begin{equation}\label{e:int_repr}
\mbi{z} (s) = e^{(s-s_0
    ) \msf{A}}\bz_0+ \int_{s_0}^{s} e^{(s-\ta
    ) \msf{A}} \mbi{h}_{3} (\ta, \mbi{z} (\ta)) \, d \ta.
\end{equation}
Differentiating this with respect to $s_0$ up to three times and evaluating at $s=s_0$ yields
\begin{align*}
\partial_{s_0}\bz(s_0) & =-\msf{A}\bz_0-\bh_3(s_0,\bz_0) \\
\partial^2_{s_0}\bz(s_0) & =  \msf{A}^2\bz_0+\msf{A}\bh_3(s_0, \bz_0) - \partial_s\bh_3(s_0, \bz_0) - D\bh_3(s_0,\bz_0) \partial_{s_0} \bz(s_0)  \\
\partial^3_{s_0}\bz(s_0) & = -\msf{A}^3\bz_0 - \msf{A}^2\bh_3(s_0, \bz_0) + 2\msf{A} \partial_s\bh_3(s_0, \bz_0) - \partial_s^2\bh_3(s_0, \bz_0) + \msf{A} D\bh_3(s_0,\bz_0) \partial_{s_0} \bz(s_0) \\
& \qquad - D(\partial_s\bh_3)(s_0,\bz_0) \partial_{s_0} \bz(s_0)  - D\bh_3(s_0,\bz_0) \tdf{}{s_0}\partial_{s_0}\bz(s_0)   - D^2\bh_3(s_0,\bz_0) (\partial_{s_0} \bz(s_0),\partial_{s_0} \bz(s_0)) \notag \\
& \qquad - D\bh_3(s_0,\bz_0)  \partial^2_{s_0} \bz(s_0). \notag
\end{align*}
Clearly, for $1\leqs l\leqs 3$,
\begin{equation*}
\norm{\partial^l_{s_0}\bz(s_0)}\leqs C \ve.
\end{equation*}
Such initial conditions are needed for analyzing the variational equations
\begin{align} 	\label{e:z_der_s_s0}
\frac{d}{ds}\partial_{s_0}\bz & = \left(\msf{A}+D\bh_3(s,\bz)\right)\partial_{s_0}\bz
\\
\frac{d}{ds}\partial^2_{s_0}\bz & = \left(\msf{A}+D\bh_3(s,\bz)\right)\partial^2_{s_0}\bz + D^2\bh_3(s,\bz)(\partial_{s_0}\bz,\partial_{s_0}\bz)	\label{e:z_der_s_s0s0}
\\
\frac{d}{ds}\partial^3_{s_0}\bz & = \left(\msf{A}+D\bh_3(s,\bz)\right)\partial^3_{s_0}\bz + 3D^2\bh_3(s,\bz)(\partial_{s_0}^2\bz,\partial_{s_0}\bz) + D^3\bh_3(s,\bz)(\partial_{s_0}\bz,\partial_{s_0}\bz,\partial_{s_0}\bz).\notag
\end{align}
One then checks recursively, using \eqref{e:radial_ode} and Corollary~\ref{c:gronwall}, that
\begin{equation*}
\norm{\partial^l_{s_0}\bz(s)}\leqs C \ve e^{\frac{\lambda_1}{2}(s-s_0)} \leqs C \ve
\end{equation*}
hold for $1\leqs l\leqs 3$ and $s-s_0\leqs K$.

Equations \eqref{e:z_der_s_s0} and \eqref{e:z_der_s_s0s0} provide us with an expression for $\partial_{s}\partial^l_{s_0}\bz(s)$ with $l=1$ and $l=2$. Moreover, \eqref{e:z_der_s_s0} can be differentiated with respect to $s$ to yield an expression for $\partial_{s}^2\partial_{s_0}\bz(s)$. The bounds in \eqref{e:z_kick_der} are then readily obtained.

Finally, we will prove \eqref{e:z_kick_der_z0}. To this end, notice that
\beq\label{e:bz_der_bz0_ode}
\frac{d}{ds}\partial_{\bz_0}\bz = \left(\msf{A} + D\bh_3(s,\bz)\right)\partial_{\bz_0}\bz.
\eeq
In particular, each row, $\partial_{\bz_{0,i}}\bz$, of $\partial_{\bz_0}\bz$ satisfies this equation. Hence, by principle \eqref{e:radial_ode},
$
\frac{d}{ds}\norm{\partial_{\bz_{0,i}}\bz} \leqs \frac{\lambda_1}{2} \norm{\partial_{\bz_{0,i}}\bz},
$
so that the matrix $\partial_{\bz_0}\bz$ remains perpetually bounded. Integrating both sides of \eqref{e:bz_der_bz0_ode} from $s_0$ to $s$ and recalling $\partial_{\bz_0}\bz (s_0)=\msf{1}$ gives
\beqn
\norm{\partial_{\bz_0}\bz(s)-\msf{1}}\leqs |s-s_0| \left(|\lambda_{n-1}|+\sup_{s_0\leqs
    s'\leqs s}\norm{D\bh_3(s',\bz)}
\sup_{s_0\leqs s'\leqs s}\norm{\partial_{\bz_0}\bz(s')} \right),
\eeqn
where $\norm{\,\cdot\,}$ now denotes the matrix norm induced by the Euclidean norm. This estimate implies \eqref{e:z_kick_der_z0}.
\end{proof}


\begin{proof}[Proof of Proposition \ref{p:f_per_dif}]
Throughout the proof, $\norm{\,\cdot\,}_{C^3}$ will stand for the $C^3$-norm with respect to $s_0$.

By \eqref{e:nf_full_t} and \eqref{e:int_nf_l}, $\tilde s$ and $\hat s$ have to satisfy
\begin{equation}\label{e:s_rho}
 \int_{s_{0}}^{\tilde{s}} b_{0} (\ta) \, d \ta = \rho =  \int_{s_{0}}^{\hat{s}} b_{0}(\tau) + v(\tau)\, d \ta,
\end{equation}
where $v(s)=\mbi{b}_{1} ^{\msf{T}}(s) \msf{P} (s) \mbi{z}(s) + \mcal{O}_{s,\mbi{z}} (\ve)
+ \mcal{O}_{s,\mbi{z}} 
(\ve \norm{\mbi{z}(s)}) + \mcal{O}_{s,\mbi{z}}(\norm{\mbi{z}(s)}^{2})$ and
$\mbi{z}=\mbi{z}(s)$ 
solves \eqref{e:nf_full_z} with $\mbi{z}(s_0)=\mbi{z}_0$. 

We use the implicit function theorem to find $\tilde s$. Clearly, $F:\mbb{R}\times \mbb{R}\to \mbb{R}:(s_0,s)\mapsto  \int_{s_{0}}^{s} b_{0} (\ta) \, d \ta - \rho $ is $C^3$. Observe that $F(s_0,s_0)=-\rho$ and $\lim_{s\to\infty} F(s_0,s)=\infty$ as $\min b_0=m>0$. By the intermediate value theorem, there exists a number $\tilde s$ such that $F(s_0,\tilde s)=0$. Because $\frac{\partial}{\partial s}F(s_0,s)=b_0(s)\geqs m$, the implicit function theorem implies that $\tilde s$ is a $C^3$-function of $s_0$. Notice that $F(s_0+1,s+1)\equiv F(s_0,s)$, so that $\tilde s(s_0+1)=\tilde s(s_0)+1$ which implies that $s_0\mapsto \tilde s(s_0)-s_0$ is periodic. 

Now that we have $\tilde s$, let us define the function
\begin{equation*}
g(\xi) \defas -\rho+\int_{s_{0}}^{\tilde s+\xi} b_{0} (\ta) + v(\tau)\, d \ta.
\end{equation*}
Notice that, denoting $\xi_1=\hat s-\tilde s$, the right side of \eqref{e:s_rho} is equivalent to $g(\xi_1)=0$. The Taylor expansion $g(\xi)=g(0)+g'(0)\xi+\delta_2g(\xi)$ yields
\begin{equation*}
G(\xi)\defas-\frac{1}{g'(0)}\left(g(0)+\delta_2g(\xi)\right)=\xi,
\end{equation*}
which we regard, for all fixed $\bz_0$, as a fixed point equation on the space of $C^3$ functions $\xi=\xi(s_0)$. Assuming $G$ is a contraction in a closed, origin-centered, ball $\bar B_r\subset C^3$ of radius $r$, there exists a unique solution, $\xi_1$, to $G(\xi)=\xi$ inside the ball. Next, we prove that for a suitably small value of $r$, $G$ is indeed a contraction.

First, notice that 
\begin{align*}
g(0) & = \int_{s_{0}}^{\tilde s} v(\tau)\, d \ta = (\tilde s-s_0) \int_0^1 v((1-\tau)s_0+\tau \tilde s)\, d \ta\\
g'(0) & = b_0(\tilde s)+v(\tilde s) \\
\delta_2 g(\xi) & = \xi^2 \int_0^1 (1-\tau)\,g''(\xi\tau)\,d\tau =  \xi^2 \int_0^1 (1-\tau)\,(b_0'+v')(\tilde s+\xi\tau)\,d\tau
\end{align*}
are smooth functions of $s_0$. Because $\tilde s$ is $C^3$ in $s_0$ and $\inf_{s_0}{g'(0)}>0$, the bounds \eqref{e:z_kick_der} yield 
\begin{equation*}
\left\|\frac{1}{g'(0)}\right\|_{C^3}\leqs C
\quad \text{and}\quad
\left\|\frac{g(0)}{g'(0)}\right\|_{C^3}\leqs C \ve.
\end{equation*}
Moreover,
\begin{equation*}
\norm{\delta_2g(\xi)}_{C^3} \leqs C\norm{\xi}_{C^3}^2 \sup_{\zeta\in \bar B_r}\norm{(b_0'+v')(\tilde s+\zeta)}_{C^3}.
\end{equation*}
Hence, $\norm{G(\xi)}_{C^3}\leqs C_0 (\ve + r^2)$ for some $C_0$. Choosing $r=2C_0 \ve$,
we have 
$G(\bar B_r)\subset \bar B_r$ for $\ve$ small enough. 

Second, let $\xi^1$ and $\xi^2$ be elements of $\bar B_r$. Since the map $\xi\mapsto G(\xi)$ is differentiable and the operator norm of the derivative obeys the bound $ \sup_{\xi\in\bar B_r}\norm{DG(\xi)}_{\mathscr{L}(C^3)}\leqs C \sup_{\xi\in\bar B_r}\norm{D\delta_2g(\xi)}_{\mathscr{L}(C^3)}\leqs Cr$, the mean value theorem yields $\norm{G(\xi^1)-G(\xi^2)}_{C^3}\leqs 
Cr \norm{\xi^1-\xi^2}_{C^3}$. Hence, $G$ is a contraction on $\bar B_r$ if $\ve$ is
sufficiently small.

We will now prove that the fixed point, $\xi_1$, of $G$ is a periodic function of $s_0$. Let us denote $\bz(s,s_0,\bz_0)$ the solution and $v(s)|_{s_0}$ the function $v$ defined above, when the initial condition $\bz(s_0)=\bz_0$ is being used. Because $\bh_3(s+2,\bz)=\bh_3(s,\bz)$ in \eqref{e:nf_full_z_v2}, we have $\bz(s+2,s_0+2,\bz_0)= \bz(s,s_0,\bz_0)$ and $v(s+2)|_{s_0+2}=v(s)|_{s_0}$.
Since $g(\xi_1)=0$ for all values of $s_0$ and $\tilde s(s_0+2)=\tilde s(s_0)+2$, the computation
\begin{align*}
g(\xi_1)(s_0+2) &= -\rho+\int_{s_0+2}^{\tilde s(s_0+2)+\xi_1(s_0+2)} b_{0} (\ta) + v(\tau)|_{s_0+2}\, d \ta =  -\rho+\int_{s_0+2}^{\tilde s(s_0)+2+\xi_1(s_0+2)} b_{0} (\ta) + v(\tau)|_{s_0+2}\, d \ta \\
&= -\rho+\int_{s_0}^{\tilde s(s_0)+\xi_1(s_0+2)} b_{0} (\ta+2) + v(\tau+2)|_{s_0+2}\, d \ta = -\rho+\int_{s_0}^{\tilde s(s_0)+\xi_1(s_0+2)} b_{0} (\ta) + v(\tau)|_{s_0}\, d \ta \\
& = g(\xi_1)(s_0)+ \int_{\tilde s(s_0)+\xi_1(s_0)}^{\tilde s(s_0)+\xi_1(s_0+2)} b_{0} (\ta) + v(\tau)|_{s_0}\, d \ta, 
\end{align*}
implies that the last integral vanishes despite the fact that the integrand is positive,
so we must have $\xi_1(s_0+2)=\xi_1(s_0)$.

As the last step, we will bound the difference $\hat \bz-\tilde \bz$. Let $\bzone$ and $\bztwo$ solve \eqref{e:nf_full_z} and \eqref{e:nf_full_z_l}, respectively, with the initial condition $\bzone(s_0)=\bztwo(s_0)=\bz_0$. Both of these are $C^3$ functions of $(s_0,\bz_0)$ by the smoothness of the vector fields. By definition, $\hat \bz = \bzone(\hat s)$ and $\tilde \bz = \bztwo(\tilde s)$. We need a bound on the $C^3$ norm of the difference $\bsym{\xi}_2 (s_0)=\hat \bz - \tilde \bz$ for fixed $\bz_0$. Notice that $\bsym{\xi}_2(s_0)=(\bzone-\bztwo)(\hat s)+(\bztwo(\hat s)-\bztwo(\tilde s))$.

Observe that the difference $\bdelta=\bzone-\bztwo$ satisfies the differential equation
\begin{equation*}
\frac{d\bdelta}{ds} = \msf{A} \bzone + \mcal{O}_{s, \bzone}
(\ve \bzone) + \mcal{O}_{s,\bzone} (\ve^{2}) + \mcal{O}_{s, \bzone} (\norm{\bzone}^{2}) = \mbi{w}.
\end{equation*}
Here $\bzone$, and hence $\mbi{w}$, is to be regarded as a predetermined function for
which we already have good bounds. Indeed, let $\mathscr{S}=\{ (s_{0}, s) : 0 \leqs s_{0}
< 2, \; s_{0} \leqs s \leqs K \}$ and $\norm{\,\cdot\,}_{C^3_\mathscr{S}}$ stand for the $C^3$ norm on this
set. According to \eqref{e:z_kick_der}, $\norm{\mbi{w}-\msf{A}
  \bzone}_{C^3_\mathscr{S}}\leqs C \ve^{2}$ whereas, recalling that all
eigenvalues of $\msf{A}$ are proportional to $\lambda_1$, $\norm{\msf{A}
  \bzone}_{C^3_\mathscr{S}}\leqs C|\lambda_1| \ve$. In other words,
$\norm{\mbi{w}}_{C^3_\mathscr{S}}\leqs C|\lambda_1| \ve$.  As
$\bdelta(s_0)=0$, we have
\begin{equation*}
(\bzone-\bztwo)(\hat s)=\bdelta(\hat s)=\int_{s_0}^{\hat s} \mbi{w}(\tau)\,d\tau.
\end{equation*}
Because $\hat s$ is $C^3$ in $s_0$, it follows that $\norm{(\bzone-\bztwo)(\hat
  s)}_{C^3}\leqs C|\lambda_1| \ve$. By \eqref{e:nf_full_z_l}, the remaining contribution reads
\begin{equation*}
\bztwo(\hat s)-\bztwo(\tilde s) = \int_{\tilde s}^{\hat s} \frac{\ve \msf{P}^{-1} (\tau) \bsym{\phi} (\mbi{0}, \tau)}{\subpsup{\psi}{n}{0} (\tau)} \,d\tau.
\end{equation*}
We have seen above that $\norm{\hat s-\tilde s}_{C^3}\leqs C \ve$, which implies 
$\norm{\bztwo(\hat s)-\bztwo(\tilde s)}_{C^3}\leqs C \ve^{2}$ and finally 
$\norm{\bsym{\xi}_2(s_0)}_{C^3}\leqs C|\lambda_1| \ve$. 
\end{proof}


\begin{remark}
It follows from the previous proof that, under the conditions of Proposition~\ref{p:f_per_dif},
\begin{equation}\label{e:hat_s_der_z0}
\norm{\partial_{\bz_0}\hat s}\leqs C.
\end{equation}
Indeed, $\partial_{\bz_0}\hat s = \partial_{\bz_0}\xi_1$, as $\partial_{\bz_0}\tilde s=0$. From the fixed point equation $\xi_1=G(\xi_1)$ we get $\partial_{\bz_0}\xi_1=(\msf{1}-DG(\xi_1))^{-1}(\partial_{\bz_0}G)(\xi_1)$ and then the claimed bound. Moreover,
\begin{equation}\label{e:hat_z_der_z0}
\norm{\partial_{\bz_0}\hat \bz-\msf{1}}\leqs C|\lambda_1|.
\end{equation}
For let $\bz(s)=\bz(s;s_0,\bz_0)$ be the solution to  \eqref{e:nf_full_z} with $\bz(s_0;s_0,\bz_0)=\bz_0$ and recall that $\hat s$ depends on $(s_0,\bz_0)$. By definition, $\hat \bz = \bz(\hat s;s_0,\bz_0)$ so that $\partial_{\bz_0}\hat \bz = \partial_{s}\bz(\hat s)\partial_{\bz_0}\hat s+\partial_{\bz_0}\bz(\hat s)=\msf{1}+\mcal{O}(\lambda_1)$ by the bounds in Lemma~\ref{le:z_kick_der}.

\end{remark}


Let us view the solution
\begin{equation}\label{e:3variables}
\bz(s)=\bz(s,\hat s,\hat\bz),\quad \quad\bz(\hat s)\equiv \hat \bz
\end{equation}
to equation \eqref{e:nf_relax_z} as a function of three variables and abbreviate
$\partial_s=\partial/\partial_s$, $\hat\partial_{i}=\partial/\partial {\hat z}_{i}$,
$\hat\partial_{i_1 \cdots i_k}=\hat\partial_{i_1}\cdots\hat\partial_{i_k}$, and
$\partial_{\hat s}=\partial/\partial \hat s$.
\begin{proposition}\label{p:derivative_bounds}
Assuming $\norm{\hat \bz}/|\lambda_1|$ is small enough, we have, for $0\leqs k+l+m\leqs 3$ and $s\geqs \hat s$, the following bounds:
\begin{align*}
\left\|\partial_s^m\partial_{\hat s}^l\bz(s)\right\| & \leqs C\norm{\hat \bz} e^{\frac{\lambda_1}{2}(s-\hat s)}\\
\left\|\partial_s^m\partial_{\hat s}^l\hat\partial_{i_1 \cdots i_k}\bz(s)\right\| & \leqs \frac{C}{|\lambda_1|^{k-1}} e^{\frac{\lambda_1}{2}(s-\hat s)} \qquad (k>0).
\end{align*}
\end{proposition}
\begin{proof}
The initial conditions $(\partial \bz/\partial{\hat \bz})(\hat s)=\msf{1}$, $\hat\partial_{ij}\mbi{z}(\hat s)=\mbi{0}$, and $\hat\partial_{ijk}\mbi{z}(\hat s)=\mbi{0}$ follow from \eqref{e:3variables}, as the $\hat\bz$-derivatives can be computed \emph{after} evaluating $\bz$ at $s=\hat s$.
Similarly, taking $\hat s$-derivatives of
\begin{equation*}
\mbi{z} (s) = e^{(s - \hat{s}) \msf{A}} \left( \hat{\mbi{z}} + \int_{\hat{s}}^{s} e^{-(\ta
    - \hat{s}) \msf{A}} \mbi{h}_{1} (\ta, \mbi{z} (\ta)) \, d \ta \right)
\end{equation*}
yields first, analogously to how the identities below \eqref{e:int_repr} were obtained,
\begin{align*}
\partial_{\hat s}\bz(\hat s) & =-\msf{A}\hat \bz-\bh_1(\hat s,\hat \bz)  \\
\partial^2_{\hat s}\bz(\hat s) & = \msf{A}^2\hat \bz + \msf{A}\bh_1(\hat s, \hat \bz) - \partial_s\bh_1(\hat s, \hat \bz) - D\bh_1(\hat s,\hat \bz) \partial_{\hat s} \bz(\hat s)  \\
\partial^3_{\hat s}\bz(\hat s) & = -\msf{A}^3\hat \bz - \msf{A}^2\bh_1(\hat s, \hat \bz) + 2\msf{A} \partial_s\bh_1(\hat s, \hat \bz) - \partial_s^2\bh_1(\hat s, \hat \bz) + \msf{A} D\bh_1(\hat s,\hat \bz) \partial_{\hat s} \bz(\hat s) \\
& \qquad - D(\partial_s\bh_1)(\hat s,\hat \bz) \partial_{\hat s} \bz(\hat s)  - D\bh_1(\hat s,\hat \bz) \tdf{}{\hat s}\partial_{\hat s}\bz(\hat s)   - D^2\bh_1(\hat s,\hat \bz) (\partial_{\hat s} \bz(\hat s),\partial_{\hat s} \bz(\hat s)) \notag \\
& \qquad - D\bh_1(\hat s,\hat \bz)  \partial^2_{\hat s} \bz(\hat s). \notag
\end{align*}
These formulas can then be differentiated with respect to $\hat \bz$ in order to find higher-order initial conditions. As $\bh_1(s,\bz)=\mcal{O}(\norm{\bz}^2)$, we obtain the following estimates:
\begin{align*}
\norm{\partial^l_{\hat s} \bz(\hat s)} & \leqs C\norm{\hat \bz} \qquad l=1,2,3\\
\norm{\partial^l_{\hat s}\hat \partial_{i}\bz(\hat s)} & \leqs C\phantom{\norm{\hat \bz}} \qquad l=1,2\\
\norm{\partial_{\hat s}\hat \partial_{ij}\bz(\hat s)} & \leqs C .
\end{align*}

Combining \eqref{e:nf_relax_z} and \eqref{e:radial_ode},
\begin{equation*}
\tdf{}{s}\norm{\mbi{z}} = \frac{\bz\cdot \msf{A}\bz+\bz\cdot \mbi{h}_{1} (s, \mbi{z})}{\norm{\bz}} \leqs \lambda_1\norm{\bz} + \norm{\mbi{h}_1 (s, \mbi{z})} \leqs \frac{\lambda_1}{2} \norm{\bz}
\end{equation*}
if $\norm{\bz}/|\lambda_1|$ is small enough.  Below, $\norm{\hat \bz}/|\lambda_1|$ will always be assumed small enough. Thus, for all $s> \hat s$,
\begin{equation}\label{e:z_relax_decay}
\norm{\bz(s)}\leqs \norm{\hat\bz} e^{\frac{\lambda_1}{2}(s-\hat s)}.
\end{equation}

Differentiating \eqref{e:nf_relax_z} with respect to various components of $\hat\bz$, we obtain the variational equations
\begin{align}
\tdf{}{s} \hat\partial_i\mbi{z} &= \left(\msf{A}+ D\mbi{h}_{1} (s, \mbi{z})\right)\hat \partial_i \mbi{z} \label{e:z_rel_var1} \\
\tdf{}{s} \hat\partial_{ij}\mbi{z} &= \left( \msf{A} + D\mbi{h}_{1} (s, \mbi{z})\right)\hat\partial_{ij} \mbi{z} +  D^2\mbi{h}_{1} (s, \mbi{z})( \hat\partial_i \mbi{z} , \hat\partial_j \mbi{z} ) \label{e:z_rel_var2}\\
\tdf{}{s} \hat\partial_{ijk}\mbi{z} &= \left( \msf{A} + D\mbi{h}_{1} (s, \mbi{z})\right)\hat\partial_{ijk} \mbi{z} +  D^2\mbi{h}_{1} (s, \mbi{z})( \hat\partial_i \mbi{z} , \hat\partial_{jk} \mbi{z} ) \label{e:z_rel_var3}\\
&\qquad  +  D^2\mbi{h}_{1} (s, \mbi{z})( \hat\partial_{k} \mbi{z} , \hat\partial_{ij} \mbi{z} ) +  D^2\mbi{h}_{1} (s, \mbi{z})( \hat\partial_j \mbi{z} , \hat\partial_{ki} \mbi{z} ) + D^3\mbi{h}_{1} (s, \mbi{z})( \hat\partial_i \mbi{z} , \hat\partial_k \mbi{z} , \hat\partial_{ij} \mbi{z} ). \notag
\end{align}

Combining \eqref{e:z_rel_var1}, \eqref{e:radial_ode}, and \eqref{e:z_relax_decay}, we have
\begin{equation}\label{e:dz_relax_decay}
\norm{\hat\partial_i\bz(s)}\leqs e^{\frac{\lambda_1}{2}(s-\hat s)}
\end{equation}
in analogy with \eqref{e:z_relax_decay}. Combining \eqref{e:z_rel_var2}, \eqref{e:radial_ode}, \eqref{e:z_relax_decay}, and \eqref{e:dz_relax_decay},
\begin{equation*}
\tdf{}{s}\norm{\hat\partial_{ij}\mbi{z}}  \leqs \frac{\lambda_1}{2} \norm{\hat\partial_{ij}\bz} + C\norm{\hat\partial_i\bz}\norm{\hat\partial_j\bz} \leqs \frac{\lambda_1}{2} \norm{\hat\partial_{ij}\bz} + Ce^{\lambda_1(s-\hat s)}.
\end{equation*}
Applying Corollary~\ref{c:gronwall},
\begin{equation}\label{e:ddz_relax_decay}
\norm{\hat\partial_{ij}\mbi{z}(s)}  \leqs \frac{C}{|\lambda_1|} e^{\frac{\lambda_1}{2}(s-\hat s)}.
\end{equation}
Similarly, combining \eqref{e:z_rel_var3}, \eqref{e:radial_ode}, \eqref{e:z_relax_decay}, \eqref{e:dz_relax_decay}, and \eqref{e:ddz_relax_decay},
\begin{equation}\label{e:dddz_relax_decay}
\norm{\hat\partial_{ijk}\mbi{z}(s)}  \leqs \frac{C}{|\lambda_1|^2} e^{\frac{\lambda_1}{2}(s-\hat s)}.
\end{equation}

Differentiating equations \eqref{e:nf_relax_z}, \eqref{e:z_rel_var1}, and
\eqref{e:z_rel_var2} with respect to $\hat s$ produces equations for $\partial^l_{\hat s}
\bz$, $ \partial^l_{\hat s}\hat \partial_{i}\bz $, and $\partial_{\hat s}\hat \partial_{ij}\bz$. For example,
\begin{align*}
\frac{d}{ds}\partial_{\hat s}\bz & = \left(\msf{A}+D\bh_1(s,\bz)\right)\partial_{\hat s}\bz
\\
\frac{d}{ds}\partial^2_{\hat s}\bz & = \left(\msf{A}+D\bh_1(s,\bz)\right)\partial^2_{\hat s}\bz + D^2\bh_1(s,\bz)(\partial_{\hat s}\bz,\partial_{\hat s}\bz).
\end{align*}
Such equations can be handled in a similar fashion and it is easy to verify that the
additional $\hat s$-derivatives do not change the bounds by more than a constant
prefactor. The bounds with $m=0$ in the proposition are now clear. The bounds with $m\neq 0$ follow immediately from the appropriate differential equation; for instance, bounding the right-hand side of \eqref{e:z_rel_var1} yields the bound on $\frac{d}{ds}\hat\partial_i\bz$.
\end{proof}


\begin{proof}[Proof of Proposition \ref{p:e_bound}]
We first view the error terms $\mscr{E}_k$ with $1\leqs k\leqs 3$ as smooth functions of $(\hat s, \hat \bz)$ and bound their derivatives with respect to these variables. The bounds on the $C^3$-norms with respect to $s_0$ follow by the chain rule. Bounding the $C^3$-norms of $\mscr{E}_4$ and $\mscr{E}_5$ is trivial and is done at the end of the proof. Throughout the proof, $\norm{\,\cdot\,}_{C^3}$ will stand for the $C^3$-norm with respect to $s_0$.

\subsection*{Terms $\mscr{E}_{1}$ and $\mscr{E}_{2}$}
It is convenient to express $\mscr{E}_{1}$ and $\mscr{E}_{2}$ in the form
\begin{align*}
\mscr{E}_{1} 
 & = \int_{0}^{\infty} \mbi{b}_{1} (\hat s+\ta)^{\msf{T}} \msf{P} (\hat s+\ta) \int_{0}^{\ta} e^{(\tau-\xi) \msf{A}} \mbi{h}_{1} (\hat s+\xi, \mbi{z} (\hat s+\xi)) \, d \xi \, d \ta \\
\mscr{E}_{2} 
& = \int_{0}^{\infty} h_{2} (\hat s + \ta, \mbi{z} (\hat s + \ta)) \, d \ta .
\end{align*}
First of all, $\norm{(\mbi{b}_1^\msf{T} \msf{P})(\hat s+\tau)}_{C^3}\leqs C\norm{\mbi{b}_1^\msf{T} \msf{P}}_{C^3}$ for every $\tau$, because $\norm{\hat s-s_0}_{C^3}\leqs C$, so that
\begin{align*}
\norm{\mscr{E}_1}_{C^3} &\leqs C\norm{\mbi{b}_1^\msf{T} \msf{P}}_{C^3} \int_{0}^{\infty}  \int_{0}^{\ta} e^{\lambda_1(\tau-\xi)} \norm{\mbi{h}_{1} (\hat s+\xi, \mbi{z} (\hat s+\xi))}_{C^3} \, d \xi \, d \ta \\
\norm{\mscr{E}_2}_{C^3} &\leqs \int_{0}^{\infty} \norm{h_{2} (\hat s + \ta, \mbi{z} (\hat s + \ta))}_{C^3} \, d \ta . 
\end{align*}
Since $\bh_1(s,\bz)$ and $h_2(s,\bz)$ are periodic in the variable $s$, their partial
derivatives of any order (less than four) with respect to $s$ are periodic functions of
$s$ and can be bounded exactly as $\bh_1(s,\bz)$ and $h_2(s,\bz)$. To save a considerable
amount of space, we write $\eta=\hat s+\xi$ and $\bze=(\hat s+\xi,\bz(\hat s+\xi))$ below.

Notice that the first three total $\hat s$-derivatives of $\bz(\eta)$ are
\begin{align*}
\frac{d}{d\hat s}\bz(\eta) &= \partial_s\bz(\eta) + \partial_{\hat s}\bz(\eta) \\
\frac{d^2}{d\hat s^2}\bz(\eta) &= \partial_s^2\bz(\eta) + 2\partial_s\partial_{\hat s} \bz(\eta) + \partial^2_{\hat s} \bz(\eta)
\\
\frac{d^3}{d\hat s^3}\bz(\eta) &= \partial_s^3\bz(\eta) + 3\partial_s^2\partial_{\hat s} \bz(\eta) + 3\partial_s\partial^2_{\hat s} \bz(\eta) + \partial^3_{\hat s} \bz(\eta).
\end{align*}
Taking $\hat\bz$-derivatives of the first two formulas above, 
\begin{align*}
\frac{d}{d\hat s}\hat\partial_i\bz(\eta) &=  \partial_s\hat\partial_i\bz(\eta) + \partial_{\hat s}\hat\partial_i\bz(\eta) \\ 
\frac{d}{d\hat s}\hat\partial_{ij}\bz(\eta) &=  \partial_s\hat\partial_{ij}\bz(\eta) + \partial_{\hat s}\hat\partial_{ij}\bz(\eta) \\
\frac{d^2}{d\hat s^2}\hat\partial_i\bz(\eta) &= \partial_s^2\hat\partial_i\bz(\eta) + \partial_s\partial_{\hat s} \hat\partial_i\bz(\eta) + \partial^2_{\hat s} \hat\partial_i\bz(\eta).
\end{align*}
Proposition~\ref{p:derivative_bounds} then implies the bounds for $0\leqs k+l\leqs 3$:
\begin{align*}
\left\|\frac{d^l}{d\hat s^l}\bz(\eta)\right\| & \leqs C\norm{\hat \bz} e^{\frac{\lambda_1}{2}\xi}\\
\left\|\frac{d^l}{d \hat s^l}\hat\partial_{i_1 \dots i_k}\bz(\eta)\right\| & \leqs \frac{C}{|\lambda_1|^{k-1}} e^{\frac{\lambda_1}{2}\xi} \qquad(k>0).
\end{align*}
These will be used to bound the $C^3$-norm of $\bh_1(\bze)$. To this end, we compute
\begin{align*}
\frac{d}{d\hat s}\mbi{h}_{1} (\bze) &= \partial_{s}\mbi{h}_{1} (\bze) + D\mbi{h}_{1} (\bze)\frac{d}{d\hat s}\bz(\eta)
\\
& = \mcal{O}\!\left(\norm{\hat \bz}^2 e^{\lambda_1\xi}\right) 
\\
\frac{d^2}{d\hat s^2}\mbi{h}_{1} (\bze) &= \partial^2_{s}\mbi{h}_{1} (\bze) + 2D(\partial_s\mbi{h}_{1}) (\bze)\frac{d}{d\hat s}\bz(\eta) + D\mbi{h}_{1} (\bze)\frac{d^2}{d\hat s^2}\bz(\eta) + D^2\mbi{h}_{1} (\bze)\left(\frac{d}{d\hat s}\bz(\eta),\frac{d}{d\hat s}\bz(\eta)\right)
\\
& = \mcal{O}\!\left(\norm{\hat \bz}^2 e^{\lambda_1\xi}\right)  
\\
\frac{d^3}{d\hat s^3}\mbi{h}_{1} (\bze) &= 
\partial^3_{s}\mbi{h}_{1} (\bze) + D(\partial^2_s\mbi{h}_{1}) (\bze)\frac{d}{d\hat s}\bz(\eta)
+ 2D(\partial_s^2\mbi{h}_{1}) (\bze)\frac{d}{d\hat s}\bz(\eta) + 3D(\partial_s\mbi{h}_{1}) (\bze)\frac{d^2}{d\hat s^2}\bz(\eta) 
\\ & \quad 
+ 3D^2(\partial_s\mbi{h}_{1}) (\bze)\left(\frac{d}{d\hat s}\bz(\eta),\frac{d}{d\hat s}\bz(\eta)\right) +  D\mbi{h}_{1} (\bze)\frac{d^3}{d\hat s^3}\bz(\eta) 
+ 3D^2\mbi{h}_{1} (\bze)\left(\frac{d}{d\hat s}\bz(\eta),\frac{d^2}{d\hat s^2}\bz(\eta)\right) 
\\ & \quad 
+ D^3\mbi{h}_{1} (\bze)\left(\frac{d}{d\hat s}\bz(\eta),\frac{d}{d\hat s}\bz(\eta),\frac{d}{d\hat s}\bz(\eta)\right)  .
\\
& = \mcal{O}\!\left(\norm{\hat \bz}^2 e^{\lambda_1\xi}\right)  
\end{align*}

\begin{align*}
\frac{d}{d\hat z_i}\mbi{h}_{1} (\bze) &= D\mbi{h}_{1} (\bze) \hat\partial_i \mbi{z} (\eta) 
\\
& = \mcal{O}\!\left(\norm{\hat \bz} e^{\lambda_1\xi}\right)  
\\
\frac{d^2}{d\hat z_id\hat z_j}\mbi{h}_{1} (\bze) &=  D\mbi{h}_{1} (\bze) \hat\partial_{ij} \mbi{z} (\eta) + D^2\mbi{h}_{1} (\bze) (\hat\partial_i \mbi{z} (\eta),\hat\partial_j \mbi{z} (\eta)) 
\\
& = \mcal{O}\!\left(\left(\frac{\norm{\hat \bz}}{|\lambda_1|}+1\right) e^{\lambda_1\xi}\right)  = \mcal{O}\!\left( e^{\lambda_1\xi}\right) 
\\
\frac{d^3}{d\hat z_id\hat z_jd\hat z_k}\mbi{h}_{1} (\bze) &= D\mbi{h}_{1} (\bze) \hat\partial_{ijk} \mbi{z} (\eta) +  D^2\mbi{h}_{1} (\bze)( \hat\partial_i \mbi{z}(\eta) , \hat\partial_{jk} \mbi{z}(\eta) )   +  D^2\mbi{h}_{1} (\bze)( \hat\partial_{k} \mbi{z}(\eta) , \hat\partial_{ij} \mbi{z}(\eta) ) 
\\
&\qquad +  D^2\mbi{h}_{1} (\bze)( \hat\partial_j \mbi{z}(\eta) , \hat\partial_{ki} \mbi{z}(\eta) ) + D^3\mbi{h}_{1} (\bze)( \hat\partial_i \mbi{z}(\eta) , \hat\partial_k \mbi{z}(\eta) , \hat\partial_{ij} \mbi{z}(\eta) ).
\\
& = \mcal{O}\!\left(\left(\frac{\norm{\hat \bz}}{|\lambda_1|^2}+\frac{1}{|\lambda_1|}\right) e^{\lambda_1\xi}\right) =  \mcal{O}\!\left(\frac{1}{|\lambda_1|}e^{\lambda_1\xi}\right)  
\end{align*}
Taking $\hat z$-derivatives of $\frac{d}{d\hat s}\bh_1(\bze)$, $\frac{d^2}{d\hat
  s^2}\bh_1(\bze)$, and the resulting expression for $\frac{d^2}{d\hat z_id\hat
  s}\bh_1(\bze)$, we get
\begin{align*}
\frac{d^2}{d\hat z_id\hat s}\mbi{h}_{1} (\bze) &= D(\partial_s\mbi{h}_{1}) (\bze) \hat\partial_i \mbi{z} (\eta) + D\mbi{h}_{1} (\bze)\frac{d}{d\hat s}\hat\partial_i \mbi{z} (\eta) + D^2\mbi{h}_{1} (\bze)\left(\frac{d}{d\hat s} \mbi{z} (\eta),\hat\partial_i \mbi{z} (\eta)\right)
\\
& =  \mcal{O}\!\left(\norm{\hat\bz}e^{\lambda_1\xi}\right)  
\\
\frac{d^3}{d\hat z_id\hat s^2}\mbi{h}_{1} (\bze) &= D(\partial_s^2\mbi{h}_{1}) (\bze) \hat\partial_i \mbi{z} (\eta) + 2D(\partial_s\mbi{h}_{1}) (\bze)\frac{d}{d\hat s}\hat\partial_i \mbi{z} (\eta) + 2D^2(\partial_s\mbi{h}_{1}) (\bze)\left(\frac{d}{d\hat s} \mbi{z} (\eta),\hat\partial_i \mbi{z} (\eta)\right)
\\
& \quad + D\mbi{h}_{1} (\bze)\frac{d^2}{d\hat s^2}\hat\partial_i \mbi{z} (\eta) + D^2\mbi{h}_{1} (\bze)\left(\frac{d^2}{d\hat s^2} \mbi{z} (\eta),\hat\partial_i \mbi{z} (\eta)\right)
\\
& \quad + 2D^2\mbi{h}_{1} (\bze)\left(\frac{d}{d\hat s} \hat \partial_i\bz(\eta),\frac{d}{d\hat s}\bz(\eta)\right)+  D^3\mbi{h}_{1} (\bze)\left(\hat \partial_i \bz(\eta),\frac{d}{d\hat s}\bz(\eta),\frac{d}{d\hat s}\bz(\eta)\right)
\\
& =  \mcal{O}\!\left(\norm{\hat\bz}e^{\lambda_1\xi}\right)  
\\
\frac{d^3}{d\hat z_id\hat z_j d\hat s}\mbi{h}_{1} (\bze) &=  D(\partial_s\mbi{h}_{1}) (\bze) \hat\partial_{ij} \mbi{z} (\eta) + D^2(\partial_s \mbi{h}_{1}) (\bze)\left(\hat\partial_i \mbi{z} (\eta),\hat\partial_j \mbi{z} (\eta)\right) + D\mbi{h}_{1} (\bze)\frac{d}{d\hat s}\hat\partial_{ij} \mbi{z} (\eta)
\\
& \quad + D^2\mbi{h}_{1} (\bze)\left(\frac{d}{d\hat s} \hat \partial_i\bz(\eta),\hat\partial_j\bz(\eta)\right) + D^2\mbi{h}_{1} (\bze)\left(\frac{d}{d\hat s} \hat \partial_j\bz(\eta),\hat\partial_i\bz(\eta)\right)
\\
& \quad + D^2\mbi{h}_{1} (\bze)\left(\frac{d}{d\hat s} \bz(\eta),\hat\partial_{ij}\bz(\eta)\right) +  D^3\mbi{h}_{1} (\bze)\left(\frac{d}{d\hat s}\bz(\eta),\hat \partial_i \bz(\eta),\hat \partial_j \bz(\eta)\right).
\\
& = \mcal{O}\!\left(\left(\frac{\norm{\hat \bz}}{|\lambda_1|}+1\right) e^{\lambda_1\xi}\right) =  \mcal{O}\!\left(e^{\lambda_1\xi}\right)  
\end{align*}

We bound the derivatives of $h_{2} (\hat s + \ta, \mbi{z} (\hat s + \ta))$ in exactly the same way.

\subsection*{Term $\mscr{E}_{3}$}
Setting $\mbi{v}(\tau)=\mbi{b}_{1} (\ta)^{\msf{T}} \msf{P} (\ta) - \bsym{\Si}$,  we have $\mscr{E}_{3} = \hat{\mbi{z}} \cdot \int_{\hat{s}}^{\infty} \mbi{v}(\tau) e^{(\ta - \hat{s}) \msf{A}} \, d \ta$.
Using the facts that $\mbi{v}$ is 2-periodic, that its integral vanishes, and that $\msf{A}$ is negative definite,
\begin{equation*}
\begin{split}
\int_{\hat{s}}^{\infty} \mbi{v}(\tau) e^{(\ta - \hat{s}) \msf{A}} \, d \ta & =
\int_{0}^{\infty} \mbi{v}(\hat s + \tau) e^{\tau \msf{A}} \, d \ta = \sum_{k=0}^\infty
\left( \int_{0}^{2} \mbi{v}(\hat s + \tau) e^{\tau \msf{A}} \, d \ta \right) e^{2k\msf{A}}\\
&=  \left( \int_{0}^{2} \mbi{v}(\hat s + \tau) e^{\tau \msf{A}} \, d \ta \right) \left(\msf{1}-e^{2\msf{A}}\right)^{-1} 
\\
&=  \left( \int_{0}^{2} \mbi{v}(\hat s + \tau) \left(e^{\tau \msf{A}} -\msf{1}\right) \,
  d \ta \right) \left(\msf{1}-e^{2\msf{A}}\right)^{-1} .
\end{split}
\end{equation*}
Hence,
\begin{equation*}
\frac{d^k}{d\hat s^k} \int_{\hat{s}}^{\infty} \mbi{v}(\tau) e^{(\ta - \hat{s}) \msf{A}} \,
d \ta  = \left( \int_{0}^{2} \mbi{v}^{(k)}(\hat s + \tau) \left(e^{\tau \msf{A}}
    -\msf{1}\right)\, d \ta \right) \left(\msf{1}-e^{2\msf{A}}\right)^{-1} .
\end{equation*}
Recalling that $\msf{A}$ is diagonal, we obtain for each value of $k$ the upper bound
\begin{equation} \label{e:per_int_bound} \left| \frac{d^k}{d\hat
      s^k}\int_{\hat{s}}^{\infty}\left( \mbi{v}(\tau) e^{(\ta - \hat{s}) \msf{A}}
    \right)_i \, d \ta \right| = \left\|v_i^{(k)}\right\|_\infty \left( \int_0^2
  \left|e^{\tau\lambda_i}-1 \right|\,d\tau \right) (1-e^{2\lambda_i})^{-1} \leqs
  C\left\|\mbi{v}^{(k)}\right\|_\infty.
\end{equation}

\subsection*{Incorporating $(s_0,\bz_0=\mbi{0})\mapsto (\hat s,\hat \bz)$}
We set  $\bz_0 = \mbi{0}$ and denote $(\hat s,\hat \bz)=H_{k} (s_{0}, \mbi{0})$.

As $\hat\bz=\bz(\hat s)=\bz(\hat s(s_0),s_0,\bzero)$, we have $\frac{d^k\hat \bz}{ds_0^k} = \frac{d^k}{ds_0^k}\bz(\hat s(s_0),s_0,\bzero)$. The bounds
\begin{equation*}
\left\|\frac{d^k\hat \bz}{ds_0^k}\right\|\ \leqs C\ve \qquad (1\leqs k\leqs 3)
\end{equation*}
follow from the fact  that $\hat s$ is a $C^3$-function of $s_0$ and the bounds in \eqref{e:z_kick_der}.

Since $\hat s$ and $\hat \bz$ are functions of $s_0$, for any function $u=u(\hat s,\hat\bz)$,
\begin{align*}
\frac{d}{ds_0} u &= \frac{d\hat s}{ds_0}\partial_{\hat s}u + \frac{d\hat z_i}{ds_0} \partial_{\hat z_i}u
\\
\frac{d^2}{ds_0^2} u &=  \frac{d^2\hat s}{ds_0^2}\partial_{\hat s}u + \frac{d^2\hat z_i}{ds_0^2} \partial_{\hat z_i}u + \left(\frac{d\hat s}{ds_0}\right)^2\partial_{\hat s \hat s}u + 2\frac{d\hat s}{ds_0} \frac{d\hat z_i}{ds_0}\partial_{\hat s \hat z_i}u+ \frac{d\hat z_i}{ds_0}\frac{d\hat z_j}{ds_0} \partial_{\hat z_i \hat z_j}u
\\ 
\frac{d^3}{ds_0^3} u &=   \frac{d^3\hat s}{ds_0^3}\partial_{\hat s}u + \frac{d^3\hat z_i}{ds_0^3} \partial_{\hat z_i}u + 2\frac{d\hat s}{ds_0}\frac{d^2\hat s}{ds_0^2}\partial_{\hat s \hat s}u + 2\left(\frac{d^2\hat s}{ds_0^2} \frac{d\hat z_i}{ds_0}+\frac{d\hat s}{ds_0} \frac{d^2\hat z_i}{ds_0^2}\right)\partial_{\hat s \hat z_i}u+ 2\frac{d^2\hat z_i}{ds_0^2}\frac{d\hat z_j}{ds_0} \partial_{\hat z_i \hat z_j}u 
\\ 
& \quad + \frac{d^2\hat s}{ds_0^2} \frac{d}{ds_0}\partial_{\hat s}u + \frac{d^2\hat z_i}{ds_0^2} \frac{d}{ds_0}\partial_{\hat z_i}u + \left(\frac{d\hat s}{ds_0}\right)^2\frac{d}{ds_0}\partial_{\hat s \hat s}u + 2\frac{d\hat s}{ds_0} \frac{d\hat z_i}{ds_0}\frac{d}{ds_0}\partial_{\hat s \hat z_i}u+ \frac{d\hat z_i}{ds_0}\frac{d\hat z_j}{ds_0}\frac{d}{ds_0} \partial_{\hat z_i \hat z_j}u.
\end{align*}
Here summation over repeated indices is understood and we leave it to the reader to expand the remaining $s_0$-derivatives on the last line. Using the bounds derived earlier, we then get
\begin{align*}
\left\|\bh_1(\bze)\right\| & \leqs C\norm{\hat \bz}^2 e^{\lambda_1\xi} \\
\left\|\frac{d}{ds_0}\bh_1(\bze)\right\| & \leqs C\left(\norm{\hat \bz}^2 + \ve\norm{\hat \bz}\right)e^{\lambda_1\xi} \\
\left\|\frac{d^2}{ds_0^2}\bh_1(\bze)\right\| & \leqs C\left(\norm{\hat \bz}^2 + \ve\norm{\hat \bz} + \ve^2\right) e^{\lambda_1\xi}  \\
\left\|\frac{d^3}{ds_0^3}\bh_1(\bze)\right\| & \leqs C\left(\norm{\hat \bz}^2 + \ve\norm{\hat \bz} + \ve^2+\frac{\ve^3}{|\lambda_1|}\right) e^{\lambda_1\xi}.
\end{align*}
Similar bounds are obtained for $h_2(\bze)$.  We conclude that
\begin{align*}
\norm{\mscr{E}_1}_{C^3} &\leqs C\norm{\mbi{b}_1^\msf{T}
  \msf{P}}_{C^3}\frac{\ve^2}{|\lambda_1|^2} \leqs C \si \frac{\ve^2}{|\lambda_1|^2} \\
\norm{\mscr{E}_2}_{C^3} &\leqs C\frac{\ve^2}{|\lambda_1|} \\
\norm{\mscr{E}_3}_{C^3} &\leqs C\ve \si.
\end{align*}
The final inequality involving $\mscr{E}_{1}$ holds because $\frac{\mbi{b}_{1}^{\msf{T}}
  \msf{P}}{\si}$ is independent of $\si$.

\subsection*{Terms $\mscr{E}_{4}$ and  $\mscr{E}_{5}$}
Writing $\mscr{E}_{4}$ in the form
\begin{equation*}
\mscr{E}_{4} = (\tilde s-\hat s)\int_{0}^{1} b_{0} ((1-\ta)\tilde s+\tau \hat s) \, d \ta
\end{equation*}
and recalling that $\tilde s$ and $\hat s$ are both $C^3$ functions of $s_0$ allows us to estimate
\begin{equation*}
\norm{\mscr{E}_{4}}_{C^3} \leqs C\norm{\tilde s-\hat s}_{C^3} \leqs C\ve.
\end{equation*}
Proposition~\ref{p:f_per_dif} was used here. Finally, by the same proposition,
\begin{equation*}
\norm{\mscr{E}_{5}}_{C^3} =
\norm{\inproa{\bsym{\xi}_{2} (s_{0}, \mbi{0})}{\bar{\mbi{d}}}}_{C^3} \leqs C\ve{\| \bsym{\Si} \|},
\end{equation*}
which finishes the proof.
\end{proof}


\begin{lemma}
For all $s_0$ and $a$, we have $\norm{\partial_{\bz_0}s_{\infty} (s_0, \bzero, a)} > 0$.
\end{lemma}

\begin{proof}
Differentiating both sides of \eqref{e:s_infty_1} with respect to $\bz_0$, we get
\begin{equation}\label{e:s_infty_init_der_1}
0=b_0(s_\infty) \partial_{\bz_0}s_\infty+ \vinproa{\bsym{\Sigma} \left( \int_{0}^{\infty}
    e^{\ta \msf{A}} \, d \ta \right)}{\partial_{\bz_0}\hat\bz} + \mbi{R},
\end{equation}
where
\begin{equation*}
\begin{split}
  \mbi{R} &= -b_0(\hat s)\partial_{\bz_0}\hat s+ \vinproa{\left(
      \int_{0}^{\infty}(\mbi{b}_{1}^{\msf{T}} \msf{P}-\bsym{\Sigma}) (\hat s+\ta) e^{\ta
        \msf{A}} \, d \ta \right)}{\partial_{\bz_0}\hat\bz}\\
  &\qquad {}+ \vinproa{\hat{\mbi{z}}}{\left( \int_{0}^{\infty} (\mbi{b}_{1}^{\msf{T}}\msf{P})' (\hat
      s+\ta) e^{\ta \msf{A}} \, d \ta \right) (\partial_{\bz_0}\hat s)} +
  \sum_{k=1}^{2} \partial_{\bz_0} \mscr{E}_{k}.
\end{split}
\end{equation*}
Because \eqref{e:per_int_bound} holds for any periodic, zero-integral function, the two
integrals appearing in $\mbi{R}$ are $\mcal{O} (\si)$ in the limit $\lambda_1\to 0$.
Terms $\partial_{\bz_0}\hat s$ and $\partial_{\bz_0}\hat\bz$ are bounded by
\eqref{e:hat_s_der_z0} and \eqref{e:hat_z_der_z0}, respectively.  Estimating
$\pdop{\mbi{z}_{0}} \mscr{E}_{1}$ and $\pdop{\mbi{z}_{0}} \mscr{E}_{2}$, we conclude that
\begin{equation*}
\norm{\mbi{R}} = \mcal{O} (\si) + \mcal{O} \left( \frac{\si \ve}{\abs{\la_{1}}^{2}}
\right).
\end{equation*}
>From \eqref{e:s_infty_init_der_1} and
\eqref{e:hat_z_der_z0}, we have
\begin{equation}\label{e:s_infty_init_der_2} 
\partial_{\bz_0}s_\infty = \frac{1}{b_0(s_\infty)} \left[ \bsym{\Sigma} 
\msf{A}^{-1}(\msf{1}+\mcal{O}(\lambda_1)) +\mcal{O}(\si) + \mcal{O} \left( \frac{\si
    \ve}{\abs{\la_{1}}^{2}} \right) \right]
\end{equation}
as $\lambda_1\to 0$. Since $\bsym{\Sigma} \msf{A}^{-1} =
(-\Sigma_i\lambda_i^{-1})_{i=1}^{n-1}$, if $\frac{\ve}{\abs{\la_{1}}}$ is sufficiently
small, then the first term on the right side of~\eqref{e:s_infty_init_der_2} dominates and
thus $\norm{\pdop{\mbi{z}_{0}} s_{\infty}} > 0$.

\end{proof}

\bibliographystyle{amsplain}
\bibliography{shear_induced_chaos}

\end{document}